\documentclass[11pt]{amsart}
\usepackage{amssymb,amscd,  latexsym, graphicx, mathrsfs, enumerate}

\usepackage{color}
\usepackage{hyperref}
\hypersetup{unicode=true,bookmarksnumbered=false,bookmarksopen=false, breaklinks=false,pdfborder={0 0 1},colorlinks=true,linkcolor=blue}
\usepackage{cleveref}
\usepackage{placeins}
\newtheorem{thm}{Theorem}[section]
\newtheorem{cor}[thm]{Corollary}
\newtheorem{lemma}[thm]{Lemma}
\newtheorem{prop}[thm]{Proposition}
\newtheorem{conj}[thm]{Conjecture}

\theoremstyle{remark}

\newtheorem{rmk}[thm]{Remark}

\theoremstyle{definition}
\newtheorem{defi}[thm]{Definition}
\numberwithin{equation}{section}

\newcommand{\sym}{\operatorname{Sym}}

\newcommand{\Hom}{\mathrm{Hom}}

\newcommand{\oo}{\mathcal{O}}

\newcommand{\Z}{\mathbb{Z}}
\newcommand{\C}{\mathbb{C}}

\newcommand{\Q}{\mathbb{Q}}

\newcommand{\R}{\mathbb{R}}
\newcommand{\A}{\mathbb{A}}
\newcommand{\tr}{\operatorname{Tr}}

\newcommand{\modulo}{\operatorname{mod}}

\newcommand{\mat}[4]{\left(\begin{matrix}#1&#2\\#3&#4\end{matrix}\right)}
\newcommand{\lrangle}[2]{\langle #1,#2\rangle}

\newcommand{\rmm}[1]{\mathrm{#1}}
\newcommand{\bff}[1]{\mathbf{#1}}

\newcommand{\set}[2]{\left\{#1\,\middle\vert\,#2\right\}}

\newcommand{\call}[1]{\mathcal{#1}}
\newcommand{\frakk}[1]{\mathfrak{#1}}

\newcommand{\mbb}[1]{\mathbb{#1}}

\newcommand{\GL}{\mathbf{GL}}

\newcommand{\SL}{\mathbf{SL}}

\begin{document}
\title[Rationality of quaternionic Eisenstein series on ${\rm U}(2,n)$]{Rationality of quaternionic Eisenstein series on ${\rm U}(2,n)$}
\author{Henry H. Kim and Yi Shan}
\date{\today}
\thanks{The first author is partially supported by NSERC grant \#482564.
}
\address{Henry H. Kim \\
Department of mathematics \\
 University of Toronto \\
Toronto, Ontario M5S 2E4, CANADA \\
and Korea Institute for Advanced Study, Seoul, KOREA}
\email{henrykim@math.toronto.edu}

\address{Yi Shan \\
Department of mathematics \\
 University of Toronto \\
Toronto, Ontario M5S 2E4, CANADA}
\email{yshan@math.utoronto.ca}



\date{\today}

\begin{abstract}
    Let $\bff{G}=\rmm{U}(2,n)$ be the unitary group associated to a Hermitian space over a quadratic imaginary number field $E$. We assume that 2 is unramified in $E$, and the Hermitian space splits at all finite places and has signature $(2,n)$, where $n\equiv 2\modulo 4$.
    A theory of Fourier expansions of quaternionic modular forms on $\bff{G}$ is developed in \cite{HMY}.
    In this paper, we define a family of degenerate Heisenberg Eisenstein series $E_{\ell}$ for $\ell>n$ on $\bff{G}$,
    which is a weight $\ell$ quaternionic modular form,
    and we explicitly compute their Fourier expansions.
    We prove that the Fourier coefficients of $E_{\ell}$ are rational in a certain sense,
    and that their denominators are uniformly bounded by an integer depending only on $\ell,n$, and $E$.
    This provides the first family of quaternionic Eisenstein series whose Fourier coefficients are known to be rational or algebraic.
\end{abstract}
\maketitle
\setcounter{tocdepth}{1}
\tableofcontents

\section{Introduction}
When the symmetric space of a reductive $\Q$-group $\bff{G}$ has a Hermitian structure,
there is a good theory of holomorphic modular forms,
\emph{i.e.\,}one may consider the automorphic forms on $\bff{G}$ whose archimedean component lies in a holomorphic discrete series representation of $\bff{G}(\R)$.
In many situations,
these holomorphic modular forms admit classical Fourier expansions,
and their Fourier coefficients often encode interesting arithmetic information.
Eisenstein series form a special family of holomorphic modular forms,
and the rationality of their Fourier coefficients is known in many cases \cite{B, H, Ka1, Sh, Ts}.
Moreover, these Fourier coefficients are known to have bounded denominators.

In \cite{GW}, Gross and Wallach singled out a special class of \emph{quaternionic discrete series representations} for certain groups $\bff{G}$,
and when it takes the place of the holomorphic discrete series above,
one obtains the \emph{quaternionic modular forms} on $\bff{G}$.
A general theory of Fourier expansions of quaternionic modular forms on quaternionic exceptional groups and $\rmm{SO}(n,4)$ are developed in \cite{Po20},
and an analogous theory for the unitary group $\rmm{U}(2,n)$ is treated in \cite{HMY}. 
In \cite{Po20b},
Pollack defines a family of quaternionic (degenerate) Heisenberg Eisenstein series in the settings of \cite{Po},
and computes some of their Fourier coefficients.

One motivation of our work is the study of a family of Eisenstein series on $\rmm{SO}(3,n+1)$ defined in \cite{Po20}.
Although they are neither holomorphic nor quaternionic,
they admit an analogous theory of Fourier expansion,
and their Fourier coefficients are proved to be rational in a certain sense in \cite{KY2}.
In this paper, we obtain a similar rationality result for the quaternionic degenerate Heisenberg Eisenstein series on a unitary group $\rmm{U}(2,n)$ with $n$ even.

We now set up the notation necessary to state \Cref{main theorem}.
For any $n\equiv 2\modulo 4$, let $(\bff{V},\lrangle{\,}{\,})$ be a Hermitian space over a quadratic imaginary number field $E$ 
that splits over all finite places and has signature $(2,n)$ at the archimedean place.
Let $\bff{G}=\rmm{U}(2,n)$ be the unitary group associated with $\bff{V}$,
and $\bff{P}=\bff{MN}$ the parabolic subgroup of $\bff{G}$ stabilizing a fixed isotropic line in $\bff{V}$.
The Levi factor $\bff{M}$ is isomorphic to $\rmm{Res}_{E/\Q}(\bff{G}_{\rmm{m}})\times \bff{U}(\bff{V}_{0})$,
where $\bff{V}_{0}$ is a Hermitian subspace of $\bff{V}$ with signature $(1,n-1)$.
The unipotent radical $\bff{N}$ is non-abelian with a one-dimensional center $\bff{Z}=[\bff{N},\bff{N}]$,
and the quotient $\bff{N}(\Q)/\bff{Z}(\Q)$ can be identified with $\bff{V}_{0}$.
The Lie group $\bff{G}(\R)$ has a maximal compact subgroup $K_{\infty}=\rmm{U}(2)\times \rmm{U}(n)$,
and for any natural number $\ell$, 
we define $\mbb{V}_{\ell}$ to be the irreducible representation $\left(\sym^{2\ell}\C^{2}\otimes\det_{\rmm{U}(2)}^{-\ell}\right)\boxtimes \bff{1}$ of $K_{\infty}$.

A \emph{weight $\ell$ quaternionic modular form on $\bff{G}$} is a smooth function $F:\bff{G}(\A)\rightarrow \mbb{V}_{\ell}$ of moderate growth such that 
$F$ is left $\bff{G}(\Q)$-invariant, right $K_{\infty}$-equivariant, and its restriction to $\bff{G}(\R)$ is killed by certain differential operators.
According to \cite{HMY},
the constant term of $F$ along $\bff{Z}$ has the following Fourier expansion: 
for any $g=g_{f}g_{\infty}\in \bff{G}(\A_{f})\times\bff{G}(\R)$,
\begin{align}
    \label{eqn Fourier QMF}
    F_{\bff{Z}}(g)=F_{\bff{N}}(g)+\sum_{\substack{0\neq T\in \bff{V}_{0}\\ \lrangle{T}{T}\geq 0}}a_{T}(F,g_{f})\call{W}_{T}(g_{\infty}),
\end{align}
where $F_{\bff{N}}$ is the constant term of $F$ along $\bff{N}$,
$a_{T}(F,g_{f}):\bff{G}(\A_{f})\rightarrow\C$ is a locally constant function,
and $\call{W}_{T}:\bff{G}(\R)\rightarrow \mbb{V}_{\ell}$ is the so-called \emph{generalized Whittaker function}.
The function $\call{W}_{T}$ plays a role similar to that of $e^{2\pi i nz}$ in the theory of holomorphic modular forms on $\SL(2,\Z)$.

In this paper, we define a $\mbb{V}_{\ell}$-valued Eisenstein series $E_{\ell}(g,s)$ associated with the (unnormalized) degenerate principal series $\rmm{Ind}_{\bff{P}(\A)}^{\bff{G}(\A)}|\nu|_{\A_{E}}^{s}$,
where $\nu$ is the projection $\bff{M}\rightarrow \rmm{Res}_{E/\Q}(\bff{G}_{\rmm{m}})$.
We show that when $s=\ell+1>n+1$,
this Eisenstein series $E_{\ell}(g):=E_{\ell}(g,s=\ell+1)$ converges absolutely and is indeed a quaternionic modular form of weight $\ell$ on $\bff{G}$.
Our main theorem concerns the Fourier expansion of $E_{\ell}(g)$:
\begin{thm}\label{main theorem} (\Cref{thm full Fourier expansion})
Assume that the discriminant of $E$ is odd.
When $\ell>n$, the Eisenstein series   
$E_\ell(g)$
has rational Fourier coefficients in the following sense:
\begin{itemize}
    \item The locally constant functions $a_{T}(E_{\ell},g_{f})$ take values in $\Q$ when restricted to $\bff{M}(\A_{f})$.
    \item The constant term $E_{\ell,\bff{N}}(g)$ is a rational multiple of $|\nu(g)|_{\A_{E}}^{\ell+1}\cdot\frac{\zeta_{E}(\ell+1)}{\pi^{2\ell+1}}[u_{1}^{n}][u_{2}^{n}]$ when restricted to $\bff{M}(\A)$, where $[u_{1}^{n}][u_{2}^{n}]$ denotes a vector in a fixed basis of $\mbb{V}_{\ell}$. 
\end{itemize}
Moreover, the value of the $T$th Fourier coefficient $a_{T}(E_{\ell},g_{f})$ at $g_{f}=1$ 
is non-zero if and only if $T$ lies in a self-dual Hermitian $\oo_{E}$-lattice $\bff{V}_{0}(\Z)$ of $\bff{V}_{0}$.
We have an explicit formula for the Fourier coefficients $a_{T}(E_{\ell},1)$,
and their denominators are all divisors of an integer depending only on $E,\ell$, and $n$.
\end{thm}
\begin{rmk}
    \label{rmk assumptions}
    Here are some remarks on our assumptions:
    \begin{enumerate}
        \item To remove the assumption on the parity of the discriminant of $E$, 
        one needs to perform the local ramified computation in \Cref{section rank 2 Fourier coefficients finite part} for $p=2$.
        \item We only consider the group $\rmm{U}(2,n)$ with $n$ even. 
        If one could obtain a result similar to \Cref{thm Siegel series ramified case} for odd $n$,
        then an explicit Fourier expansion of $E_{\ell}$ would also be obtained.
        The problem is that in this case the $T$th Fourier coefficient for any anisotropic $T\in\bff{V}_{0}$ is a rational multiple of $\frac{\pi^{2\ell-n+2}}{\zeta(2\ell-n+2)}$, and we do not even know the algebraicity of this fraction.
    \end{enumerate}
\end{rmk} 
The Fourier coefficients can be written as certain Eulerian integrals,
so the proof of \Cref{main theorem} is reduced to local computations at $g_{f}=1$.
The most complicated part is the computation of $a_{T}(E_{\ell},1)$ for anisotropic vectors $T\in\bff{V}_{0}$,
and our strategy is to consider $\bff{V}_{0}$ as a quadratic space, with symmetric bilinear form $\tr_{E/\Q}\lrangle{\,}{\,}$:
\begin{itemize}
    \item For a prime $p$, the $\Q_{p}$-component of $a_{T}(E_{\ell},1)$ is a combination of several quadratic Siegel series.
    When $p$ is unramified in $E$, we directly apply a result from \cite{KY2} on the Siegel series.
    When $p$ is ramified, we prove a ramified version of this result.
    In both cases, the $\Q_{p}$-component turns out to be the product of $(1-p^{n-2\ell-2})$ with $p^{(n-2\ell-1)v_{p}(\lrangle{T}{T})}Q_{T,p}(p^{\ell-\frac{n-1}{2}})$
    (together with some rational factor independent of $T$ when $p$ is ramified), where $Q_{T,p}$ is a polynomial in the ideal $\Z[X^{2}]+p^{1/2}X\Z[X^{2}]$ of $\Z[p^{1/2}][X]$, satisfying a certain functional equation.
    Hence, the product of all the finite components of $a_{T}(E_{\ell},1)$ is a rational multiple of $\zeta(2\ell-n+2)^{-1}\lrangle{T}{T}^{n-2\ell-1}$, and the denominator of this rational number is bounded by an integer independent of $T$.
    \item The archimedean component
    can be written as the Fourier transform (on $\bff{V}_{0}(\R)$, viewed as a quadratic space) of a complicated function $F_{0,\ell}$.
    Following the strategy of \cite{Po}, 
    we compare $F_{0,\ell}$ to the inverse Fourier transform of $\rmm{Char}(\lrangle{T}{T}>0)\lrangle{T}{T}^{2\ell-n+1}\call{W}_{T}(1)$,
    and find that the archimedean contribution to $a_{T}(E_{\ell},1)$ is the product of $\pi^{2\ell-n+2}\lrangle{T}{T}^{2\ell-n+1}$ with a rational number independent of $T$.
\end{itemize}
Combining the two parts together, 
we conclude that for an anisotropic $T$, 
the $T$th Fourier coefficient $a_{T}(E_{\ell},1)$ is a rational number with a uniformly bounded denominator. 
This gives the first rationality (or algebraicity) result for a family of quaternionic Eisenstein series.
\begin{rmk}
    \label{rmk compare to other works}
    The Fourier coefficients of the quaternionic modular forms on quaternionic exceptional groups considered in \cite{Po20b}
    are more complicated than ours, 
    because cubic structures are involved
    and they have more ranks (five) compared to three in our case.
\end{rmk} 

We organize our paper as follows. 
We first review the structure of $\bff{G}=\rmm{U}(2,n)$ in \Cref{section rank 2 unitary groups}.
In \Cref{section quaternionic Heisenberg Eisenstein series},
we recall the theory of quaternionic automorphic forms on $\bff{G}$
and define the quaternionic Heisenberg Eisenstein series $E_\ell(g,s=\ell+1)$.
The heart of the paper is \Cref{section rank 2 Fourier}, 
where we compute the rank 2 Fourier coefficients,
while the Fourier coefficients of other ranks are calculated in \Cref{section lower rank Fourier}.
Finally, 
we construct a conjectural Saito-Kurokawa type lift on $\bff{G}$ from a holomorphic modular form on $\SL_{2}(\Z)$ in \Cref{section conj Saito Kurokawa lift},
as a potential application of this paper.

\smallskip

\section{Preliminaries on the unitary groups of real rank 2}
\label{section rank 2 unitary groups}
Let $D$ be a square-free natural number with $D\equiv 3\modulo 4$,
and $E$ the imaginary quadratic number field $\Q(\sqrt{-D})$,
as a subfield of $\C$.
Denote the nonzero element in $\mathrm{Gal}(E/\Q)$ by $x\mapsto \overline{x}$,
and for any $\Q$-algebra $R$ we extend this involution to $E\otimes_{\Q}R$ and define 
$\tr_{E\otimes_{\Q}R}:E\otimes_{\Q}R\rightarrow R$ to be the trace map $x\mapsto x+\overline{x}$.

For any place $v$ of $\Q$,
let $|\,|_{\Q_{v}}$ be the normalized absolute value on the local field $\Q_{v}$.
Set $E_{v}:=E\otimes_{\Q}\Q_{v}$,
and denote by $|\,|_{E_{v}}$ the absolute value on $E_{v}$ defined by $x\mapsto |x\overline{x}|_{\Q_{v}}$.

Denote by $\A=\A_{\Q}$ (\emph{resp.\,}$\A_{E}$) the adèle ring of $\Q$ (\emph{resp.\,}$E$).
Fix an additive character $\psi=\psi_{\Q}$ on $\Q\backslash\A$ 
satisfying that for every finite prime $p$, $\psi_{p}$ is an additive character of $\Q_{p}$ with conductor $\Z_{p}$,
and at the archimedean place, $\psi_{\infty}(x)=e^{2\pi i x}$.  
Set $\psi_{E}$ to be the additive character of $E\backslash\A_{E}$
that sends $x\in E$ to $\psi(\tr_{\A_{E}}(\frac{x}{\sqrt{-D}}))$.
The additive character $\psi_{E}$ can be decomposed as $\otimes_{v}\psi_{E,v}$ with the following properties:
\begin{itemize}
    \item At the archimedean place $\infty$, $E_{\infty}=\C$ and the character $\psi_{E,\infty}$ sends $z\in \C$ to $e^{2\pi i \cdot\mathrm{Re}(z)}$.
    \item When $p=\frakk{p}_{1}\frakk{p}_{2}$ splits in $E$, $E_{p}=E_{\frakk{p}_{1}}\times E_{\frakk{p}_{2}}\simeq \Q_{p}\times \Q_{p}$, and the character $\psi_{E,\frakk{p}_{i}}$, $i=1,2$, has conductor $\oo_{E_{\frakk{p}_{i}}}\simeq \Z_{p}$.
    \item When $p$ is inert in $E$, $E_{p}$ is an unramified quadratic field extension of $\Q_{p}$, and the character $\psi_{E,p}$ has conductor $\oo_{E_{p}}$.
    \item When $p=\frakk{p}^{2}$ is ramified in $E$, $E_{p}=E_{\frakk{p}}$ is a ramified quadratic field extension of $\Q_{p}$. The additive character $\psi\circ\tr_{E_{\frakk{p}}}:E_{\frakk{p}}\rightarrow\C^{\times}$ has conductor $\frakk{p}^{-1}$,
    thus $\psi_{E,\frakk{p}}$ has conductor $\oo_{E_{\frakk{p}}}$.
\end{itemize}
\begin{rmk}\label{rmk exclude ramified 2}
    We take $D$ with $D\equiv 3\modulo 4$ to avoid $p=2$ being ramified in $E=\Q(\sqrt{-D})$.
    If one considers a general imaginary quadratic number field $E$,
    the additive character $\psi_{E}$ should be replaced by $\psi(\tr_{\A_{E}}(\frac{x}{\sqrt{\Delta_{E}}}))$,
    where $\Delta_{E}$ is the discriminant of $E$.
\end{rmk}

\subsection{Hermitian spaces}
\label{section Hermitian spaces}
For any natural number $n\equiv 2 \modulo 4$,
let $(\bff{V},\lrangle{\,}{\,})$ be a non-degenerate Hermitian space over $E$ with signature $(2,n)$
such that $\bff{V}$ is split over all the finite places,
where the Hermitian form is conjugate-linear in the second variable.
The existence of such a Hermitian space is guaranteed by the Hasse principle (for example, see \cite[Chap. 10]{Scharlau}).
For any $\Q$-algebra $R$,
we denote by $\bff{V}(R)$ the Hermitian module $\bff{V}\otimes_{\Q}R$ over $E\otimes_{\Q}R$. 
Note that $\bff{V}=\bff{V}(\Q)$.

Fix a pair of isotropic vectors $(b_{1},b_{2})$ in $\bff{V}$
such that $\lrangle{b_{1}}{b_{2}}=1$,
and take $\bff{V}_{0}$ to be the orthogonal complement of $Eb_{1}\oplus Eb_{2}$ in $\bff{V}$,
which is a non-degenerate Hermitian space of signature $(1,n-1)$.
We fix again a pair of isotropic vectors $(c_{1},c_{2})$ in $\bff{V}_{0}$ with $\lrangle{c_{1}}{c_{2}}=1$,
then $\bff{V}$ is the orthogonal direct sum of $Eb_{1}\oplus Eb_{2}\oplus Ec_{1}\oplus Ec_{2}$ and a Hermitian space $\bff{V}_{00}$ with signature $(0,n-2)$.

Take $u_{1}=\frac{1}{\sqrt{2}}(b_{1}+b_{2})$, $u_{2}=\frac{1}{\sqrt{2}}(c_{1}+c_{2})$, $v_{n}=\frac{1}{\sqrt{2}}(b_{1}-b_{2})$ and $v_{n-1}=\frac{1}{\sqrt{2}}(c_{1}-c_{2})$,
and they form an orthogonal basis of $\C b_{1}\oplus \C b_{2}\oplus \C c_{1}\oplus \C c_{2}\subseteq V:=\bff{V}(\R)$.
Fix a basis $\{v_{1},\ldots,v_{n-2}\}$ of $V_{00}:=\bff{V}_{00}(\R)$ such that 
$\lrangle{v_{i}}{v_{j}}=-\delta_{ij}$ for any $1\leq i,j\leq n-2$.
We have obtained a basis of $V$ with the following properties:
\begin{itemize}
    \item $\lrangle{u_{i}}{u_{j}}=\delta_{ij}$, for any $1\leq i,j\leq 2$;
    \item $\lrangle{v_{i}}{v_{j}}=-\delta_{ij}$, for any $1\leq i,j\leq n$;
    \item $\lrangle{u_{i}}{v_{j}}=0$ for any $1\leq i\leq 2$ and $1\leq j\leq n$;
    \item $\{u_{2},v_{1},\ldots,v_{n-1}\}$ is a basis for $V_{0}:=\bff{V}_{0}(\R)$.
\end{itemize}
Set $V_{2}^{+}=\C u_{1}\oplus \C u_{2}$ and $V_{n}^{-}=\C v_{1}\oplus \cdots\oplus \C v_{n}$,
and we have a decomposition $V_{2}^{+}\oplus V_{n}^{-}$ of $V$ into definite subspaces.

The Hermitian space $\bff{V}_{0}$ over $E$ admits a self-dual lattice $L$,
\emph{i.e.\,}
\[L=\set{v\in \bff{V}_{0}}{\lrangle{v}{w}\in \oo_{E}\text{ for any }w\in L}.\]
We fix this lattice,
and by an abuse of notation,
denote it by $\bff{V}_{0}(\Z)$,
and the self-dual lattice $L\otimes_{\Z}\Z_{p}$ in $\bff{V}_{0}(\Q_{p})$ by $\bff{V}_{0}(\Z_{p})$.

\subsection{Unitary group \texorpdfstring{${\rm U}(2,n)$}{} and its Heisenberg parabolic subgroup}
\label{section unitary group and its Heisenberg parabolic}
Define $\bff{G}=\bff{U}(\bff{V})={\rm U}(2,n)$ to be the unitary group associated with the Hermitian space $(\bff{V},\langle{\,}{\,}\rangle)$,
and write the $\bff{G}$-action on $\bff{V}$ as a right action.
By the locally split assumption on $\bff{V}$,
$\bff{G}$ is quasi-split over $\Q_{p}$ for all $p$.

The Heisenberg parabolic subgroup $\bff{P}$ of $\bff{G}$ is defined as the stabilizer of the isotropic line spanned by $b_{1}$,
and it admits the following Levi decomposition $\bff{P}=\bff{MN}$:
\begin{itemize}
    \item The Levi factor $\bff{M}$ is chosen to be the stabilizer in $\bff{P}$ of the isotropic line $Eb_{2}$. 
    By its action on $Eb_{1}\oplus \bff{V}_{0}\oplus Eb_{2}$,
    one identifies this subgroup with $\rmm{Res}_{E/\Q}(\bff{G}_{\rmm{m}})\times\bff{U}(\bff{V}_{0})$,
    so that $(z,h)\in \rmm{Res}_{E/\Q}(\bff{G}_{\rmm{m}})\times \bff{U}(\bff{V}_{0})\simeq \bff{M}$ acts on $\bff{V}$ by
    \[(xb_{1}+w+yb_{2})\cdot(z,h)=z^{-1}xb_{1}+wh+\overline{z}yb_{2},\text{ for any }x,y\in\bff{G}_{\rmm{a}},\,w\in \bff{V}_{0}.\]
    \item The unipotent radical $\bff{N}$ is a two-step nilpotent group, and it is isomorphic to the group 
    \[\set{(v,\lambda)\in \bff{V}_{0}\times \rm{Res}_{E/\Q}\bff{G}_{\rm{a}}}{\overline{\lambda}=-\lambda},\]
    equipped with the following multiplication:
    \[(v_{1},\lambda_{1})\cdot(v_{2},\lambda_{2}):=\left(v_{1}+v_{2},\lambda_{1}+\lambda_{2}-\frac{\lrangle{v_{1}}{v_{2}}-\lrangle{v_{2}}{v_{1}}}{2}\right).\]
    The action of the corresponding element $n(v,\lambda)\in\bff{N}$ on $\bff{V}$ is given as: 
    \[b_{1}\mapsto b_{1},\,b_{2}\mapsto (-\frac{1}{2}\lrangle{v}{v}+\lambda)b_{1}+b_{2}+v,\,w\in \bff{V}_{0}\mapsto -\lrangle{w}{v}b_{1}+w.\]
    The center of $\bff{N}$ is $\bff{Z}:=\set{n(0,\lambda)}{\overline{\lambda}=-\lambda}$, which has $\Q$-dimension $1$.
\end{itemize}
The conjugation action of $\bff{M}$ on $\bff{N}$ is given by 
\[(z,h)n(v,\lambda)(z^{-1},h^{-1})=n(\overline{z}vh^{-1},z\overline{z}\lambda).\]
Denote by $\nu$
the similitude character of $\bff{P}$: 
\[\nu\left((z,h)n(v,\lambda)\right)=z\in \rmm{Res}_{E/\Q}\bff{G}_{\rmm{m}},\text{ for any }(z,h)\in\bff{M},\,n(v,\lambda)\in\bff{N},\]
then we have $b_{1}p=\nu(p)^{-1}b_{1}$ for any $p\in \bff{P}$.

Using the unitary character $\psi_{E}:E\backslash\A_{E}\rightarrow\mbb{S}^{1}\subset\C^{\times}$,
we identify $\bff{V}_{0}$ with $\Hom([\bff{N}],\mbb{S}^{1})$
by sending a vector $T\in\bff{V}_{0}$ to the unitary character
\[\chi_{T}(n(v,\lambda)):=\psi_{E}(\lrangle{T}{v})=\psi\left(\tr_{\A_{E}} \left(\frac {\lrangle{T}{v}}{\sqrt{-D}}\right)\right).\]

\section{Quaternionic Heisenberg Eisenstein series}
\label{section quaternionic Heisenberg Eisenstein series}
\subsection{Quaternionic modular forms on \texorpdfstring{${\rm U}(2,n)$}{}}
\label{section quaternionic modular forms on unitary groups}
We first recall from \cite[\S 3]{HMY} the theory of quaternionic modular forms on the unitary group $\bff{G}={\rm U}(2,n)$.

In \Cref{section Hermitian spaces},
we have a decomposition $V=V_{2}^{+}\oplus V_{n}^{-}$ of $V=\bff{V}(\R)$ into definite subspaces.
The group $K_{\infty}=\rmm{U}(V_{2}^{+})\times\rmm{U}(V_{n}^{-})\simeq \rmm{U}(2)\times \rmm{U}(n)$
is a maximal compact subgroup of $\bff{G}(\R)$.
For any integer $\ell\geq 1$,
consider the following irreducible representation of $K_{\infty}$:
\[\mbb{V}_{\ell}:=\left(\sym^{2\ell}V_{2}^{+}\otimes {\det}^{-\ell}_{\bff{U}(V_{2}^{+})}\right)\boxtimes\bff{1}.\]
Fix a basis $\{[u_{1}^{\ell-v}][u_{2}^{\ell+v}],\,-\ell\leq v\leq \ell\}$ of $\mbb{V}_{\ell}$, where $[u_{i}^{m}]:=\frac{u_{i}^{m}}{m!}\in \sym^{m}V_{2}^{+}$, $i=1,2$, $0\leq m\leq 2\ell$.

When $\ell\geq n$,
Gross and Wallach \cite{GW} construct an irreducible unitary discrete series representation $\Pi_{\ell}$ of $\bff{G}(\R)$,
which contains $\Bbb V_{\ell}$ as its minimal $K_{\infty}$-type with multiplicity $1$,
and two Schmid operators $\call{D}_{\ell}^{\pm}$ associated with $\Pi_{\ell}$ are constructed in \cite[\S 3.1]{HMY}.
\begin{defi}\label{def quaternionic modular form on unitary groups}
    \cite[Definition 3.1]{HMY}
    A \emph{weight $\ell$ quaternionic modular form on $\bff{G}={\rm U}(2,n)$} is a smooth function 
    $F:\bff{G}(\A)\rightarrow \mbb{V}_{\ell}$ of moderate growth such that:
    \begin{enumerate}
        \item For any $g\in \bff{G}(\A)$ and $\gamma\in \bff{G}(\Q)$, $F(\gamma g)=F(g)$ .
        \item For every $k\in K_{\infty}$ and $g\in\bff{G}(\A)$, $F(gk)=F(g)\cdot k$.
        \item The functions $\call{D}_{\ell}^{\pm}\left(F|_{\bff{G}(\R)}\right)$ vanish identically on $\bff{G}(\R)$.
    \end{enumerate}
\end{defi}
Let $F$ be a weight $\ell$ quaternionic modular form on $\bff{G}$.
If $T\in\bff{V}_{0}$, then the \emph{$T$-th Fourier coefficient of $F$ along $\bff{N}$} is 
\[F_{\bff{N},T}:\bff{G}(\A)\rightarrow \mbb{V}_{\ell},\,g\mapsto \int_{[\bff{N}]}F(ng)\overline{\chi_{T}(n)}dn.\]
We have the following result on the Fourier expansion of $F$:
\begin{thm}\label{thm Fourier expansion theorem of quaternionic modular forms}
    \cite[Corollary 3.7]{HMY}
    Let $F:\bff{G}(\A)\rightarrow\mbb{V}_{\ell}$ be a weight $\ell$ quaternionic modular form,
    and $F_{\bff{Z}}$ (\emph{resp.\,}$F_{\bff{N}}$) the constant term of $F$ along $\bff{Z}$ (\emph{resp.\,}$\bff{N}$).
    There exist locally constant functions 
    \[\{a_{T}(F,\cdot):\bff{G}(\A_{f})\rightarrow \C\}_{T\in\bff{V}_{0},\,\lrangle{T}{T}\geq 0}\]
    such that for any $g_{f}\in \bff{G}(\A_{f})$ and $g_{\infty}\in \bff{G}(\R)$ one has:
    \begin{align}\label{eqn general Fourier expansion of modular form}
        F_{\bff{Z}}(g_{f}g_{\infty})=F_{\bff{N}}(g_{f}g_{\infty})+\sum_{\substack{0\neq T\in \bff{V}_{0}\\\lrangle{T}{T}\geq 0}}a_{T}(F,g_{f})\call{W}_{T}(g_{\infty}),
    \end{align}
    where $\call{W}_{T}$ is the $\mbb{V}_{\ell}$-valued function on $\bff{G(\R)}$ given by the formula
    \begin{align}\label{eqn generalized Whittaker function}
        \call{W}_{T}\left(nmk\right)=\chi_{T,\infty}(n)\sum_{v=-\ell}^{\ell}|\nu(m)|^{2\ell+2}\left(\frac{|\beta_{T}(m)|}{\beta_{T}(m)}\right)^{v}K_{v}\left(|\beta_{T}(m)|\right)[u_{1}^{\ell-v}][u_{2}^{\ell+v}]\cdot k,
    \end{align}
    for any $n\in\bff{N}(\R),m\in\bff{M}(\R)$ and $k\in K_{\infty}$.
    Here $K_{v}$ denotes the K-Bessel function $K_{v}(x)=\displaystyle\frac 12\int_{0}^{\infty}t^{v-1}e^{-x(t+t^{-1})}dt$ and 
    $\beta_{T}:\bff{M}(\R)\rightarrow\C$ is defined by $\beta_T(z,h)=4\sqrt{2}\pi \lrangle{u_{2}}{zT\cdot h}$.

    Moreover, 
    If $F$ is cuspidal, then the Fourier expansion of $F_{\mathbf{Z}}$ takes the form
    $$F_Z(g_{f}g_\infty)=\sum_{T\in \bff{V}_0\atop \langle T,T\rangle> 0} a_T(F,g_{f})\mathcal W_T(g_\infty).$$
\end{thm}
\begin{rmk}
\label{rmk different additive characters}
Notice that in \cite{HMY}
the function $\beta_{T}(z,h)$ is defined as $2\sqrt{2}\pi\lrangle{u_{2}}{zT\cdot h}$.
This is because their additive character $\psi_{E}$ is chosen to be $\psi_{E}(x)=\psi(\frac{1}{2}\tr_{\A_{E}}(x))$.
\end{rmk}
\begin{rmk}
    \label{rmk structure constant term}
    A result on the structure of the constant term $F_{\bff{N}}$ similar to \cite[Theorem 1.2.1(2)]{Po20}
    is also expected for quaternionic modular forms on $\rmm{U}(2,n)$
    :for $m_{\infty}\in \rmm{U}(1,n-1)\subseteq \bff{M}(\R)$,
    $F_{\bff{N}}(m_{\infty})$ should be $\Phi(m_{\infty})[u_{1}^{2\ell}]+\beta[u_{1}^{\ell}][u_{2}^{\ell}]+\Phi^{\prime}(m_{\infty})[u_{2}^{2\ell}]$,
    where $\beta$ is a constant, $\Phi$ is associated to a holomorphic modular form on $\rmm{U}(1,n-1)$ and $\Phi^{\prime}$ is a certain $K_{\infty}\cap \bff{M}(\R)$-right translate of $\Phi$.
    We will see in \Cref{section constant term of Eisenstein series} that holomorphic modular forms on $\rmm{U}(1,n-1)$ have no contribution to the constant term along $\bff{N}$ of our Heisenberg Eisenstein series. 
\end{rmk}

\subsection{The Heisenberg Eisenstein series \texorpdfstring{$E_{\ell}$}{}}
\label{subsection definition of Heisenberg Eisenstein}
For $\ell\ge n$, the quaternionic discrete series $\Pi_{\ell}$ can be embedded into the (unnormalized) degenerate principal series $\rm{Ind}_{\bff{P}(\R)}^{\bff{G}(\R)}|\nu|_{\C}^{\ell+1}$ \cite[Theorem 9]{Wallach}
(by our notation $|z|_{\C}=z\overline{z}=|z|^{2}$),
thus we now construct an Eisenstein series on $\bff{G}$, 
which is a weight $\ell$ quaternionic modular form,
from some section of $\rmm{Ind}_{\bff{P}(\A)}^{\bff{G}(\A)}|\nu|_{\A_{E}}^{s}$, $s\in \C$.

For each finite prime $p$,
let $\Phi_{p}$ be the characteristic function of the lattice spanned by $\bff{V}_{0}(\Z_{p})$ and $b_{1},b_{2}$,
and $\Phi_{f}:=\otimes_{p<\infty}\Phi_{p}$.
We then choose a section $f_{\ell}(g_{f}g_{\infty},s)=f_{fte}(g_{f},s)f_{\ell,\infty}(g_{\infty},s)$ 
of $\rmm{Ind}_{\bff{P}(\A)}^{\bff{G}(\A)}|\nu|_{\A_{E}}^{s}$ as follows:
\begin{itemize}
    \item The archimedean part $f_{\ell,\infty}$ is $\bff{V}_{\ell}$-valued, and given by the formula:
    \[f_{\ell,\infty}((z,h)k,s)=\pi^{-2\ell-1}|z|^{2s}\left([u_{1}^{\ell}][u_{2}^{\ell}]\right)\cdot k,\text{ for any }(z,h)\in \bff{M}(\R),\,k\in K_{\infty}.\]
    \item The finite part $f_{fte}$ is defined by 
    $f_{fte}(g_{f},s)=\int_{\GL_{1}(\A_{E,f})}|t|_{\A_{E,f}}^{s}\Phi_{f}(t\cdot b_{1}g_{f})dt$.
\end{itemize}
We set the (Heisenberg) Eisenstein series 
\[E_{\ell}(g,s)=\sum_{\gamma\in \bff{P}(\Q)\backslash \bff{G}(\Q)}f_{\ell}(\gamma g,s).\]
Since the modulus character is $\delta_{\bff{P}}(t,m)=|t|_{\Bbb A_E}^{n+1}$ for $(t,m)\in \bff M(\Bbb A_E)$, 
by Godement's criterion on the convergence of Eisenstein series \cite[Lemma 11.1]{Bo}, 
$E_{\ell}(g,s)$ converges absolutely for $Re(s)>n+1$.
Moreover, when $s=\ell+1$ lies in the convergence range, we have:
\begin{prop}
    \label{prop at some point Eisenstein series is modular form}
    For $\ell> n$,
    the Eisenstein series $E_{\ell}(g,s=\ell+1)$ is a weight $\ell$ quaternionic modular form on $\bff{G}={\rm U}(2,n)$.
\end{prop}
\begin{proof}
    The only condition in \Cref{def quaternionic modular form on unitary groups} that is not obvious is the third one.
    It suffices to show that the restriction of $f_{\ell,\infty}$ to $\bff{M}(\R)$ is killed by the Schmid operators $\call{D}_{\ell}^{\pm}$.

    Write an element $(z,h)\in\bff{M}(\R)$ as $m=(h,r,\theta)$ so that $z=re^{i\theta},\,r\in\R_{>0}$ and $\theta\in [0,2\pi)$.
    Using this coordinates,
    we define a function $F\left(h,r,\theta\right)=r^{2s}$ on $\bff{M}(\R)$
    so that $f_{\ell,\infty}(m)=F(m)[u_{1}^{\ell}][u_{2}^{\ell}]$.
    By \cite[Proposition 3.10]{HMY},
    $f_{\ell,\infty}$ is killed by $\call{D}_{\ell}^{\pm}$ if and only if 
    $(r\partial r-2(\ell+1))F\equiv 0$ on $\bff{M}(\R)$,
    which holds when $s=\ell+1$.
\end{proof}
\subsection{Abstract Fourier expansion}
\label{section abstract Fourier expansion}
Now we give an ``abstract'' Fourier expansion of $E_{\ell}(g,s)$.
\begin{lemma}\label{lemma double coset of parabolic of unitary groups}
    The right $\bff{P}(\Q)$-space $\bff{P}(\Q)\backslash \bff{G}(\Q)$,
    the space of isotropic lines in $\bff{V}$,
    has exactly $3$ orbits of $\bff{P}(\Q)$,
    represented respectively by $E b_{1}$, $E c_{1}$ and $E b_{2}$.
\end{lemma}
\begin{proof}
    It can be shown directly by the explicit action given in \Cref{section unitary group and its Heisenberg parabolic}.
\end{proof}
Set $\bff{G}(\Q)=\bigsqcup\limits_{i=0}^{2}\bff{P}(\Q)w_{i} \bff{P}(\Q)$,
such that $w_{0}=1$,
$b_{1}\cdot w_{1}=c_{1}$ and $b_{1}\cdot w_{2}=b_{2}$,
and write the Heisenberg Eisenstein series as 
\[E_{\ell}(g,s)=\sum_{i=0}^{2}E_{\ell,i}(g,s),\text{ where }E_{\ell,i}(g,s)=\sum_{\gamma\in \bff{P}(\Q)\backslash \bff{P}(\Q)w_{i}\bff{P}(\Q)}f_{\ell}(\gamma g,s).\]
\begin{lemma}
    \label{lemma rank decomposition of Eisenstein series}
    For $\rmm{Re}(s)>n+1$, 
    one has the following expressions for $E_{\ell,i}(g,s),\,i=0,1,2$:
    \begin{enumerate}
        \item $E_{\ell,0}(g,s)=f_{\ell}(g,s)$.
        \item Let $\call{L}_{0}$ be the set of non-zero isotropic lines in $\bff{V}_{0}$.
        For any $L\in\call{L}_{0}$, select an element $\gamma(L)\in \bff{G}(\Q)$ such that $b_{1}\cdot \gamma(L)\in L$.
        Then \[E_{\ell,1}(g,s)=\sum_{L\in\call{L}_{0}}\sum_{\mu\in L^{\perp}\bff{Z}(\Q)\backslash\bff{N}(\Q)}f_{\ell}(\gamma(L)\mu g,s),\]
        where we identify the group $\set{n(v,0)\in\bff{N}(\Q)}{v\in L^{\perp}}$ with the subspace $L^{\perp}$ of $\bff{V}_{0}$. 
        \item One has 
        \[E_{\ell,2}(g,s)=\sum_{\mu\in \bff{N}(\Q)}f_{\ell}(w_{2}\mu g,s).\]
    \end{enumerate}
\end{lemma} 
\begin{proof}
    The expressions for $E_{\ell,0}$ and $E_{\ell,2}$ are clear.
    By our choice of $w_{1}$,
    the set $\bff{P}(\Q)\backslash \bff{P}(\Q)w_{1}\bff{P}(\Q)$ decomposes as the disjoint union:
    \[\bigsqcup_{L\in\call{L}_{0}}\gamma(L)\left(\rmm{Stab}_{\bff{N}(\Q)}(L)\backslash \bff{N}(\Q)\right).\]
    For an isotropic line $L=Ew$ in $\bff{V}_{0}$,
    an element $n(v,\lambda)\in\bff{N}(\Q)$ maps $L$ to the line generated by $w-\lrangle{w}{v}b_{1}$,
    thus the stabilizer of $L$ in $\bff{N}(\Q)$ is $\set{n(v,\lambda)\in\bff{N}(\Q)}{v\in L^{\perp}}\simeq L^{\perp}\bff{Z}(\Q)$.
\end{proof}

For any vector $T\in \bff{V}_{0}$,
set 
\[E_{\ell}^{T}(g,s)=\int_{\bff{N}(F)\backslash\bff{N}(\A)}\chi_{T}^{-1}(n)E_{\ell}(ng,s)dn\]
and
\[E_{\ell,i}^{T}(g,s)=\int_{\bff{N}(F)\backslash\bff{N}(\A)}\chi_{T}^{-1}(n)E_{\ell,i}(ng,s)dn,\text{ for }i=0,1,2.\]
The following lemma shows that $E_{\ell,i}^{T}$ is Eulerian.
\begin{lemma}
    \label{lemma Eulerian Fourier coefficients}
    \begin{enumerate}
        \item For any isotropic line $L\in \call{L}_{0}$,
        define $\bff{N}_{L}:=L^{\perp}\bff{Z}\subseteq \bff{N}$.
        For a non-zero isotropic vector $T\in \bff{V}_{0}$, 
        set $\bff{N}_{T}:=\bff{N}_{L_{T}}$,
        where $L_{T}$ is the line generated by $T$.
        We have:
        \begin{itemize}
            \item If $T\in\bff{V}_{0}$ is anisotropic, then $E_{\ell,1}^{T}$ is identically zero. 
            \item If $T$ is non-zero and isotropic,
            then 
            \[E_{\ell,1}^{T}(g,s)=\int_{\bff{N}_{T}(\A)\backslash \bff{N}(\A)}\chi_{T}^{-1}(n)f_{\ell}(\gamma(L_{T})ng,s)dn.\]
            \item If $T=0$,
            then 
            \[E_{\ell,1}^{0}(g,s)=\sum_{L\in \call{L}_{0}}\int_{\bff{N}_{L}(\A)\backslash \bff{N}(\A)}f_{\ell}(\gamma(L)ng,s)dn.\]
        \end{itemize}
        \item For any $T\in \bff{V}_{0}$, one has 
        \[E_{\ell,2}^{T}(g,s)=\int_{\bff{N}(\A)}\chi_{T}^{-1}(n)f_{\ell}(w_{2}ng,s)dn.\]
    \end{enumerate}
\end{lemma}
\begin{proof}
    It suffices to prove the $i=1$ case.
    For any vector $T\in \bff{V}_{0}$, we have:
    \begin{align*}
        E_{\ell,1}^{T}(g,s)=&\sum_{L\in\call{L}_{0}}\int_{[N]}\chi_{T}^{-1}(n)\left(\sum_{\mu\in\bff{N}_{L}(\Q)\backslash \bff{N}(\Q)}f_{\ell}(\gamma(L)\mu ng,s)\right)dn\\
        =&\sum_{L\in\call{L}_{0}}\int_{\bff{N}_{L}(\Q)\backslash\bff{N}(\A)}\chi_{T}^{-1}(n)f_{\ell}(\gamma(L)ng,s)dn\\
        =&\sum_{L\in\call{L}_{0}}\int_{\bff{N}_{L}(\A)\backslash\bff{N}(\A)}\left(\int_{[\bff{N}_{L}]}\chi_{T}^{-1}(r)dr\right)\chi_{T}^{-1}(n)f_{\ell}(\gamma(L)ng,s)dn\\
        =&\sum_{L\in\call{L}_{0},\,\chi_{T}|_{\bff{N}_{L}}\equiv 1}\int_{\bff{N}_{L}(\A)\backslash \bff{N}(\A)}\chi_{T}^{-1}(n)f_{\ell}(\gamma(L)ng,s)dn.
    \end{align*}
    The lemma then follows from the fact that $\chi_{T}|_{\bff{N}_{L}}\equiv 1$ if and only if $T\in L$, \emph{i.e.\,}$T=0$ or $T$ is a non-zero isotropic vector generating $L$. 
\end{proof}

\section{Rank 2 Fourier coefficients of \texorpdfstring{$E_{\ell}(g,s=\ell+1)$}{}}
\label{section rank 2 Fourier}
We first look at the most complicated part of the Fourier expansion of the Eisenstein series $E_{\ell}(g,s=\ell+1)$:
the rank $2$ part $E_{\ell}^{T}(g,s)$,
where $T$ is an anisotropic vector in $\bff{V}_{0}$.
By \Cref{lemma Eulerian Fourier coefficients},
we know that $E_{\ell}^{T}=E_{\ell,2}^{T}$.
In this section, we study $E_{\ell,2}^{T}$ for \emph{any} $T\in\bff{V}_{0}$.

For any $g\in\bff{G}(\A)$, $m=(z,h)\in \bff{M}(\A)$ and $T\in \bff{V}_{0}$,
we have 
\begin{align*}
    E_{\ell,2}^{T}(mg,s)=& \int_{\bff{N}(\A)}\chi_{T}^{-1}(n)f_{\ell}(w_{2}nmg,s)dn\\
    =&\int_{\bff{N}(\A)}\chi_{T}^{-1}(mnm^{-1})f_{\ell}(w_{2}(mnm^{-1})mg,s)d(mnm^{-1})\\
    =&\int_{\bff{N}(\A)}\chi_{zT\cdot h}^{-1}(n)|\nu(w_{2}mw_{2}^{-1})|_{\A_{E}}^{s}f_{\ell}(w_{2}ng,s)|\nu(m)|_{\A_{E}}^{n+1}dn\\
    =&|\nu(m)|_{\A_{E}}^{n+1-s}E_{\ell,2}^{zT\cdot h}(g,s).
\end{align*}
In particular,
when $s=\ell+1$ we have 
\begin{align}\label{eqn tranform of rank 2 Fourier}
    E_{\ell,2}^{T}(mg,s=\ell+1)=|\nu(m)|_{\A_{E}}^{n-\ell}E_{\ell,2}^{\widetilde{T}}(g,s=\ell+1),\,\widetilde{T}=zT\cdot h.    
\end{align}
So for the value of $E_{\ell,2}^{T}$ at $g\in\bff{M}(\A_{f})\times\bff{G}(\R)$, 
it suffices to study $E_{2}^{T}(s):=E_{\ell,2}^{T}(1,s)$,
and we decompose $E_{2}^{T}(s)$ as $\prod_{v}E_{2,v}^{T}(s)$,
where 
\[E_{2,v}^{T}(s):=\int_{\bff{N}(\Q_{v})}\chi_{T}^{-1}(n)f_{\ell,v}(w_{2}n,s)dn.\]

\subsection{Rank 2 Fourier coefficients: non-archimedean components}
\label{section rank 2 Fourier coefficients finite part}
For any finite prime $p$, 
using the explicit action of $\bff{P}(\Q_{p})$ on $\bff{V}(\Q_{p})$ given in \Cref{section unitary group and its Heisenberg parabolic},
we obtain that:
\begin{align*}
    E_{2,p}^{T}(s)&=\int_{\bff{N}(\Q_{p})}\chi_{T}^{-1}(n)f_{p}(w_{2}n,s)dn\\
    &=\int_{\bff{N}(\Q_{p})}\int_{t\in E_{p}^{\times}}\chi_{T}^{-1}(n)|t|_{E_{p}}^{s}\Phi_{p}(t\cdot b_{1}w_{2}n)dtdn\\
    &=\int_{v\in \bff{V}_{0}(\Q_{p})}\int_{\substack{x\in E_{p}\\ \overline{x}=-x}}\int_{t\in E_{p}^{\times}}\chi_{T}^{-1}(v)|t|_{E_{p}}^{s}\Phi_{p}\left(t\left(\left(-\frac{\lrangle{v}{v}}{2}+x\right)b_{1}+b_{2}+v\right)\right)dtdxdv.
\end{align*}
Let $\oo_{E_{p}}$ be $\Z_{p}\otimes_{\Z}\oo_{E}\subseteq E_{p}$.
By our choice,
$\Phi_{p}$ is the characteristic function of the lattice 
\[\oo_{E_{p}}b_{1}\oplus\oo_{E_{p}}b_{2}\oplus\bff{V}_{0}(\Z_{p}),\]
thus the function $\Phi_{p}$ in $E_{2,p}^{T}(s)$ equals $1$ 
if and only if 
\[t\in \oo_{E_{p}},\,v\in t^{-1}\bff{V}_{0}(\Z_{p}),\text{and }-\frac{\lrangle{v}{v}}{2}+x\in t^{-1}\oo_{E_{p}}.\]

\begin{lemma}\label{lemma finite rank 2 coefficients}
    \begin{enumerate}
        \item If $p$ splits in $E$,
        then 
        \[E_{2,p}^{T}(s)=\sum_{r_{1},r_{2}\geq 0}p^{-(r_{1}+r_{2})s+\min(r_{1},r_{2})}\left(\int_{v\in (p^{-r_{1}},\,p^{-r_{2}})\bff{V}_{0}(\Z_{p})}\chi_{T}^{-1}(v)\rmm{Char}(p^{\rmm{max}(r_{1},r_{2})}\lrangle{v}{v}\in \Z_{p})dv\right).\]
        \item If $p$ is inert in $E$,
        then
        \[E_{2,p}^{T}(s)=\sum_{r\geq 0}p^{-2rs+r}\left(\int_{v\in p^{-r}\bff{V}_{0}(\Z_{p})}\chi_{T}^{-1}(v)\rmm{Char}\left(p^{r}\lrangle{v}{v}\in \Z_{p}\right)dv\right).\]
        \item If an odd prime $p$ is ramified in $E$, then 
$$
    E_{2,p}^{T}(s)=\sum_{r\geq 0}p^{-rs+\lceil r/2 \rceil}\left(\int_{v\in \varpi^{-r}\bff{V}_{0}(\Z_{p})}\chi_{T}^{-1}(v)\rmm{Char}(p^{\lfloor r/2\rfloor}\lrangle{v}{v}\in\Z_{p})dv\right),
$$
where $\varpi$ is a uniformizer of $\oo_{E_{p}}$ satisfying $\tr(\varpi)=0$ and $\varpi^{2}\in p\Z_{p}^{\times}$.
    \end{enumerate}
\end{lemma}
\begin{proof}
    (1) If $p$ splits in $E$, then
we identify $E_{p}$ with $\Q_{p}\times\Q_{p}$, and rewrite $E_{2,p}^{T}(s)$ as:
\begin{align*}
   \sum_{r_{1}\geq 0,\,r_{2}\geq 0}p^{-(r_{1}+r_{2})s}\int_{v\in (p^{-r_{1}},\,p^{-r_{2}})\bff{V}_{0}(\Z_{p})}\chi_{T}^{-1}(v)\left(\int_{\substack{x\in (p^{-r_{1}},\,p^{-r_{2}})\oo_{E_{p}}+\lrangle{v}{v}/2\\ x+\overline{x}=0}}dx\right)dv.
\end{align*}
If we write $x=(y,-y)\in E_{p}$,
then $x\in (p^{-r_{1}},p^{-r_{2}})\oo_{E_{p}}+\lrangle{v}{v}/2$ is equivalent to 
\[y\in p^{-r_{1}}\Z_{p}+\lrangle{v}{v}/2\text{ and }-y\in p^{-r_{2}}\Z_{p}+\lrangle{v}{v}/2.\]
There are such $x$ only if $\lrangle{v}{v}\in p^{-\max(r_{1},r_{2})}\Z_{p}$.
Without loss of generality, assume that $r_{1}\geq r_{2}$.
Then $y=p^{-r_{2}}\widetilde{y}-\lrangle{v}{v}/2$ for some $\widetilde{y}\in \Z_{p}$.
Since the measure on the line $\{x=(y,-y)\in E_{p}\}$ is the normalized Haar measure $dy$ on $\Q_{p}$,
one has $dy=p^{r_{2}}d\widetilde{y}$,
and 
\[\int_{(y,-y)\in p^{-r_{1}}\Z_{p}\times p^{-r_{2}}\Z_{p}+\lrangle{v}{v}/2}dy=\int_{\widetilde{y}\in \Z_{p}}p^{r_{2}}d\widetilde{y}=p^{r_{2}}=p^{\min(r_{1},r_{2})}.\]
Hence,
\[E_{2,p}^{T}(s)=\sum_{r_{1},r_{2}\geq 0}p^{-(r_{1}+r_{2})s+\min(r_{1},r_{2})}\left(\int_{v\in (p^{-r_{1}},\,p^{-r_{2}})\bff{V}_{0}(\Z_{p})}\chi_{T}^{-1}(v)\rmm{Char}(p^{\rmm{max}(r_{1},r_{2})}\lrangle{v}{v}\in \Z_{p})dv\right).\]

(2) If $p$ is inert in $E$,
then $E_{p}$ is an unramified quadratic field extension of $\Q_{p}$, 
and 
\begin{align*}
    E_{2,p}^{T}(s)=\sum_{r\geq 0}p^{-2rs}\int_{v\in p^{-r}\bff{V}_{0}(\Z_{p})}\chi_{T}^{-1}(v)\int_{x\in E_{p},\,\tr(x)=0}\rmm{Char}\left(p^{r}\left(-\frac{\lrangle{v}{v}}{2}+x\right)\in\oo_{E_{p}}\right)dxdv.
\end{align*}
When $p$ is odd, 
choose a unit $u\in \oo_{E_{p}}^{\times}$ with zero trace,
then $\{x\in E_{p} |\,\tr(x)=0\}$ is $\Q_{p}u$,
with the normalized Haar measure.
For any $v\in p^{-r}\bff{V}_{0}(\Z_{p})$ and $y\in \Q_{p}$, 
$p^{r}\left(-\frac{\lrangle{v}{v}}{2}+yu\right)\in\oo_{E_{p}}$
if and only if $\lrangle{v}{v}\in p^{-r}\Z_{p}$ and $y\in p^{-r}\Z_{p}$,
thus  
\begin{align*}
    E_{2,p}^{T}(s)&=\sum_{r\geq 0}p^{-2rs}\left(\int_{v\in p^{-r}\bff{V}_{0}(\Z_{p})}\chi_{T}^{-1}(v)\rmm{Char}(p^{r}\lrangle{v}{v}\in \Z_{p})dv\int_{y\in p^{-r}\Z_{p}}dy\right)\\
    &=\sum_{r\geq 0}p^{-2rs+r}\left(\int_{v\in p^{-r}\bff{V}_{0}(\Z_{p})}\chi_{T}^{-1}(v)\rmm{Char}(p^{r}\lrangle{v}{v}\in\Z_{p})dv\right).
\end{align*}
When $p=2$, choose a unit $u\in \oo_{E_{2}}^{\times}$ with $u^{2}-u+1=0$,
then $\set{x\in E_{2}}{\tr(x)=0}$ is $\Q_{2}(1-2u)$, with the normalized Haar measure.
For $v\in 2^{-r}\bff{V}_{0}(\Z_{p})$ and $y\in \Q_{2}$,
$2^{r}\left(-\frac{\lrangle{v}{v}}{2}+y(1-2u)\right)\in \oo_{E_{p}}$ 
if and only if $2^{r}\lrangle{v}{v}\in \Z_{2}$ and $y\in \frac{\lrangle{v}{v}}{2}+2^{-r}\Z_{2}\subseteq 2^{-r-1}\Z_{2}$,
thus
\begin{align*}
    E_{2,p=2}^{T}(s)=&\sum_{r\geq 0}2^{-2rs}\left(\int_{v\in 2^{-r}\bff{V}_{0}(\Z_{p})}\chi_{T}^{-1}(v)\rmm{Char}(2^{r}\lrangle{v}{v}\in \Z_{2})dv\int_{y\in \frac{\lrangle{v}{v}}{2}+2^{-r}\Z_{2}}dy\right)\\
    =&\sum_{r\geq 0}2^{-2rs+r}\left(\int_{v\in 2^{-r}\bff{V}_{0}(\Z_{p})}\chi_{T}^{-1}(v)\rmm{Char}(2^{r}\lrangle{v}{v}\in \Z_{2})dv\right).
\end{align*}

(3) If $p$ is ramified in $E$, then $E_{p}$ is a ramified quadratic field extension of $\Q_{p}$.
As $p\neq 2$, we choose a uniformizer $\varpi$ of $\oo_{E_{p}}$ such that $\varpi+\overline{\varpi}=0$ and $\varpi^{2}\in p\Z_{p}^{\times}$,
so $\set{x\in E_{p}}{\tr(x)=0}$ is $\Q_{p}\varpi$,
with the normalized Haar measure.
The integral for $E_{2,p}^{T}(s)$ can be rewritten as:
$$
E_{2,p}^{T}(s)=\sum_{r\geq 0}p^{-rs}\int_{v\in \varpi^{-r}\bff{V}_{0}(\Z_{p})}\chi_{T}^{-1}(v)\left(\int_{x\in\Q_{p}}\rmm{Char}\left(x\varpi-\frac{\lrangle{v}{v}}{2}\in \varpi^{-r}\oo_{E_{p}}\right)dx\right)dv.
$$
Suppose that $x\varpi-\frac{\lrangle{v}{v}}{2}=\varpi^{-r}(y+\varpi z)$ for $y,z\in \Bbb Z_p$. 
If $r$ is even, one has $x\varpi-\frac{\lrangle{v}{v}}{2}=(p^{-\frac{r}{2}}z)\varpi+p^{-\frac{r}{2}}y$, 
which implies that $\lrangle{v}{v}\in p^{-\frac{r}{2}}\Bbb Z_p$ and $x\in p^{-\frac{r}{2}}\Bbb Z_p$. 
If $r$ is odd, one has 
$x\varpi-\frac{\lrangle{v}{v}}{2}=(p^{-\frac{r+1}{2}}y)\varpi+p^{-\frac{r-1}{2}}z$, 
which implies that $\lrangle{v}{v}\in p^{-\frac{r-1}{2}}\Bbb Z_p$ and $x\in p^{-\frac{r+1}{2}}\Bbb Z_p$. 
Hence,
\[
E_{2,p}^T=\sum_{r\geq 0}p^{-rs}p^{\lceil r/2\rceil}\int_{v\in \varpi^{-r}\bff{V}_{0}(\Z_{p})}\chi_{T}^{-1}(v)\rmm{Char}(p^{\lfloor r/2\rfloor}\lrangle{v}{v}\in \Z_{p})\, dv.
\qedhere
\]
\end{proof}
\subsubsection{Siegel Series for quadratic spaces}
\label{section Siegel series quadratic}
To calculate the inner integrals in \Cref{lemma finite rank 2 coefficients},
we will reduce the problem to the orthogonal case. 
When $p$ is unramified in $E$, 
we use the following result by Kim and Yamauchi:
\begin{thm}\label{thm Siegel series orthogonal case}
    \cite[Theorem 4.2]{KY2}
    Let $\psi$ be an additive unitary character on $\Q_{p}$ with conductor $\Z_{p}$.
    Let $(L,q)$ be a split quadratic lattice over $\Z_{p}$ of rank $2m$,
    with
    \[q(x_{1},\ldots,x_{m},y_{1},\ldots,y_{m})=x_{1}y_{1}+\cdots+x_{m}y_{m},\]
    and denote the associated bilinear form by $(\,,\,)$.
    For any integer $r\geq 0$ and $\eta\in L\otimes_{\Z_{p}}\Q_{p}$ with $q(\eta)\neq 0$,
    set
    \[B_{r,\eta}:=\int_{x\in p^{-r}L}\psi\left((\eta,x)\right)\rmm{Char}(p^{r}q(x)\in \Z_{p})dx.\]
    When $\eta\notin L$, $B_{r,\eta}=0$ for any $r\geq 0$.
    When $\eta\in L$,
    if $p^{k}\| q(\eta)$,
    then
    \[\sum_{r=0}^{\infty}p^{-rs}B_{r,\eta}=(1-p^{m-1-s})P_{\eta,p}(p^{m-s}),\]
    where $P_{\eta,p}(X)\in\Z[X]$ is a monic polynomial of degree $k$ satisfying the functional equation
    \begin{align}\label{eqn functional eqn split case}
         X^{k}P_{\eta,p}(X^{-1})=P_{\eta,p}(X).
    \end{align}
\end{thm}
If $p$ is an odd prime ramified in $E$,
we need the following similar result:
\begin{thm}
    \label{thm Siegel series ramified case}
    Let $\psi$ be an additive unitary character on $\Q_{p}$ with conductor $\Z_{p}$.
    Let $(L,q)$ be a quadratic lattice over $\Z_{p}$ of rank $4m$,
    with 
    \[q(x_{1},\ldots,x_{2m},y_{1},\ldots,y_{2m})=\sum_{i=1}^{m}x_{i}y_{i}+p\sum_{i=m+1}^{2m}x_{i}y_{i},\]
    and denote the associated bilinear form by $(\,,\,)$.
    For any integer $r\geq 0$ and $\eta\in L\otimes_{\Z_{p}}\Q_{p}$ with $q(\eta)\neq 0$,
    set 
    \[C_{r,\eta}:=\int_{x\in p^{-r}L}\psi((\eta,x))\rmm{Char}(p^{r}q(x)\in \Z_{p})dx.\]

    For any $\eta=(a_{1},\ldots,a_{2m},b_{1},\ldots,b_{2m})\in L\otimes_{\Z_{p}}\Q_{p}$ with $q(\eta)\neq 0$,
    set 
    \[\eta_{1}=(a_{1},\ldots,a_{m},b_{1},\ldots,b_{m}),\,\eta_{2}=(a_{m+1},\ldots,a_{2m},b_{m+1},\ldots,b_{2m}),\]
    and 
    \[k_{1}:=\min(v_{p}(\eta_{1}),v_{p}(\eta_{2})),\,k_{2}:=\min(v_{p}(\eta_{1}),v_{p}(\eta_{2})+1),\,k:=v_{p}(q(\eta)),\,k^{\prime}:=k-k_{1}-k_{2}.\]
    The Siegel series $\sum_{r\geq 0}p^{-rs}C_{r,\eta}$ is non-zero only if $k_{2}\geq 0$.
    If $k_{2}=0$ and $k_{1}=-1$,
    then $\sum_{r\geq 0}p^{-rs}C_{r,\eta}=1$.
    If $k_{2}\geq k_{1}\geq 0$, \emph{i.e.\,}$\eta\in L$,
    then 
    \begin{align*}
        \sum_{r\geq 0}p^{-rs}C_{r,\eta}=(1-p^{2m-s-1})R_{\eta,p}(p^{2m-s})+\sum_{r=0}^{k_{2}}p^{r(4m-s-1)}-p^{3m-s-1}\sum_{r=0}^{k_{1}}p^{r(4m-s-1)},
    \end{align*}
    where 
    {\small \begin{align*}
        R_{\eta,p}(X)=p^{m}\left(\sum_{r=1}^{k_{2}}\frac{p^{r(2m-1)-1}}{p^{2m-1}-1}X^{r}+\frac{p^{(k_{1}+1)(2m-1)}-1}{p^{2m-1}-1}\sum_{r=k_{2}+1}^{k_{2}+k^{\prime}}X^{r}+\sum_{r=k_{2}+k^{\prime}+1}^{k}\frac{p^{(k-r+1)(2m-1)}-1}{p^{2m-1}-1}X^{r}\right)
    \end{align*}}
    is a polynomial in $\Z[X]$ with degree $k$ satisfying the functional equation
    \begin{align}\label{eqn functional equation ramified case}
        X^{k+1}R_{\eta,p}(X^{-1})=R_{\eta,p}(X).
    \end{align}
\end{thm}
\begin{proof}
The proof is similar to that of \cite[Theorem 4.2]{KY2}.
We write $C_{r,\eta}$ as an exponential sum:
\begin{align*}
    C_{r,\eta}=\sum_{\substack{u_{1},\ldots,u_{2m},v_{1},\ldots,v_{2m}\in \Z/p^{r}\Z\\ q(u_{1},\ldots,u_{2m},v_{1},\ldots,v_{2m})\equiv 0 \modulo p^{r}}} \exp\left(\frac{2\pi i}{p^{r}}\left(\sum\limits_{i=1}^{m}(a_{i}v_{i}+b_{i}u_{i})+p\sum\limits_{i=m+1}^{2m}(a_{i}v_{i}+b_{i}u_{i})\right)\right)
\end{align*}
Using the identity 
\[\rmm{Char}\left(y\equiv 0 \modulo p^{r}\right)=p^{-r}\sum_{x\in \Z/p^{r}\Z}e^{\frac{2\pi i}{p^{r}}xy},\text{ for any }y\in \Z/p^{r}\Z,\]
we have:
\begin{align*}
    C_{r,\eta}&=p^{-r}\sum_{x\in \Z/p^{r}\Z}\sum_{u_{1},\ldots,u_{2m},v_{1},\ldots,v_{2m}\in \Z/p^{r}\Z} \exp\left( \frac{2\pi i }{p^{r}}\left(\sum\limits_{i=1}^{m}(a_{i}v_{i}+b_{i}u_{i}+x u_{i}v_{i})+p\sum\limits_{i=m+1}^{2m}(a_{i}v_{i}+b_{i}u_{i}+xu_{i}v_{i})\right)\right)\\
    &=p^{-r}\sum_{x\in \Z/p^{r}\Z}\prod_{i=1}^{m}\left(\sum_{u_{i},v_{i}\in \Z/p^{r}\Z}e^{\frac{2\pi i}{p^{r}}(xu_{i}v_{i}+a_{i}v_{i}+b_{i}u_{i})}\right)\prod_{i=m+1}^{2m}\left(\sum_{u_{i},v_{i}\in \Z/p^{r}\Z}e^{\frac{2\pi i}{p^{r-1}}(xu_{i}v_{i}+a_{i}v_{i}+b_{i}u_{i})}\right)\\
    &=:p^{-r}\sum_{x\in \Z/p^{r}\Z}C_{r,\eta,x}.
\end{align*}
Let $x=p^{j}x_{0}$ with $p\nmid x_{0}$ and $x_{0}^{-1}$ is taken modulo $p^{r-j}$.
By \cite[Lemma 11.1]{KY2},
{\Small \begin{align*}
    C_{r,\eta,x}=\left\{\begin{array}{ll}
        p^{4rm}, &\text{ if }j\geq r, a_{i}\equiv b_{i}\equiv 0 \modulo p^{r}, a_{i+m}\equiv b_{i+m}\equiv 0 \modulo p^{r-1}, 1\leq i\leq m\\
        p^{(2r+2j+1)m}e^{-\frac{2\pi i q(\eta)x_{0}^{-1}}{p^{r+j}}}, &\text{ if }0\leq j<r, a_{i}\equiv b_{i}\equiv 0\modulo p^{j}, 1\leq i\leq 2m\\
        0, &\text{ otherwise.}
    \end{array}\right.
\end{align*}}
From this formula,
we see that $C_{r,\eta}\neq 0$ only if $k_{2}\geq 0$.
When $\eta\notin L$ and $k_{2}\geq 0$,
one has $v_{p}(\eta_{2})=-1$ and $v_{p}(\eta_{1})\geq 0$.
In this case,
$C_{r,\eta}=0$ for any $r>0$ and $C_{0,\eta}=1$,
thus the Siegel series $\sum_{r\geq 0}p^{-rs}C_{r,\eta}=1$.
From now on, we assume that $\eta\in L$, and compute explicitly $C_{r,\eta}$.

If $r\leq k_{2}$,
then $v_{p}(q(\eta))\geq 2k_{2}-1\geq 2r-1$, thus
\begin{align*}
    C_{r,\eta}&=p^{-r}\left(p^{4rm}+\sum_{j=0}^{r-1}p^{(2r+2j+1)m}\phi(p^{r-j})\right)\\
    &=p^{r(4m-1)}+\sum_{j=0}^{r-1}\left(p^{(2r+2j+1)m-j}-p^{(2r+2j+1)m-j-1}\right).
\end{align*}

If $r>k_{2}$, then $C_{r,\eta,x}\neq 0$ only if $v_{p}(x)\leq k_{1}$, and 
\begin{align*}
    C_{r,\eta}=p^{-r}\sum_{j=0}^{k_{1}}p^{(2r+2j+1)m}\sum_{c\in (\Z/p^{r-j}\Z)^{\times}}e^{\frac{2\pi i c}{p^{r-j}}\cdot \frac{q(\eta)}{p^{2j}}},\text{ where }p^{2j}\mid q(\eta)\text{ for any }0\leq j\leq k_{1}. 
\end{align*}
We have 
\begin{align*}
    \sum_{c\in (\Z/p^{r-j}\Z)^{\times}}e^{\frac{2\pi i c}{p^{r-j}}\cdot \frac{q(\eta)}{p^{2j}}}=
    \sum_{\zeta\in\mu_{p^{r-j}}\text{ primitive}}\zeta^{p^{k-2j}}
    =\left\{\begin{array}{ll}
        p^{r-j}-p^{r-j-1}, & \text{ if }r+j\leq k \\
        -p^{r-j-1}, & \text{ if }r+j=k+1\\
        0, &\text{ if }r+j>k+1.
    \end{array}\right.
\end{align*}

If $k_{2}<r\leq k_{2}+k^{\prime}$,
then 
\begin{align*}
    C_{r,\eta}=p^{-r}\sum_{j=0}^{k_{1}}p^{(2r+2j+1)m}(p^{r-j}-p^{r-j-1})=\sum_{j=0}^{k_{1}}\left(p^{(2r+2j+1)m-j}-p^{(2r+2j+1)m-j-1}\right).
\end{align*}

If $k_{2}+k^{\prime}<r\leq k+1$,
then 
\begin{align*}
    C_{r,\eta}&=p^{-r}\left(\sum_{j=0}^{k-r}p^{(2r+2j+1)m}(p^{r-j}-p^{r-j-1})-p^{(2r+2(k-r+1)+1)m}p^{r-(k-r+1)-1}\right)\\
    &=\sum_{j=0}^{k-r}\left(p^{(2r+2j+1)m-j}-p^{(2r+2j+1)m-j-1}\right)-p^{(2k+3)m+r-k-2}.
\end{align*}

If $r>k+1$, then $C_{r,\eta}=0$.

Now we decompose $\sum_{r\geq 0}p^{-rs}C_{r,\eta}$ into $3$ parts:
the sum from $0$ to $k_{2}$, from $k_{2}+1$ to $k_{2}+k^{\prime}$,
and from $k_{2}+k^{\prime}+1$ to $k+1$.
 
For the sum from $0$ to $k_{2}$,
we have 
\begin{align*}
    \sum_{r=0}^{k_{2}}p^{-rs}C_{r,\eta}=\sum_{r=0}^{k_{2}}p^{r(4m-s-1)}+\sum_{0\leq j<r\leq k_{2}}\left(p^{r(2m-s)+j(2m-1)+m}-p^{r(2m-s)+j(2m-1)+m-1}\right).
\end{align*}
The second summation over $j<r$ can be written as 
\begin{align*}
    \sum_{j=0}^{k_{2}-1}\left(-p^{j(4m-s-1)+3m-s-1}+\sum_{r=j+1}^{k_{2}-1}p^{r(2m-s)+j(2m-1)+m}(1-p^{2m-s-1})+p^{k_{2}(2m-s)+j(2m-1)+m}\right).
\end{align*}
Hence $\sum\limits_{r=0}^{k_{2}}p^{-rs}C_{r,\eta}$ is equal to 
\begin{align*}
    \sum_{r=0}^{k_{2}}p^{r(4m-s-1)}(1-p^{3m-s-1})+p^{k_{2}(2m-s)+m}\sum_{j=0}^{k_{2}}p^{j(2m-1)}+(1-p^{2m-s-1})R_{1}(p^{2m-s}),
\end{align*}
where 
\[R_{1}(X)=\sum_{r=0}^{k_{2}-1}p^{m}\left(\sum_{j=0}^{r-1}p^{j(2m-1)}\right)X^{r}-p^{k_{2}(2m-1)+m}X^{k_{2}}\in \Z[X].\]

For the sum from $k_{2}+1$ to $k_{2}+k^{\prime}$ (when $k^{\prime}>0$),
we have 
\begin{align*}
    \sum_{r=k_{2}+1}^{k_{2}+k^{\prime}}p^{-rs}C_{r,\eta}=&\sum_{r=k_{2}+1}^{k_{2}+k^{\prime}}\sum_{j=0}^{k_{1}}\left(p^{r(2m-s)+j(2m-1)+m}-p^{r(2m-s)+j(2m-1)+m-1}\right)\\
    =&\sum_{j=0}^{k_{1}}p^{j(2m-1)+m}\sum_{r=k_{2}+1}^{k_{2}+k^{\prime}}\left(p^{r(2m-s)}-p^{r(2m-s)-1}\right).
\end{align*}
The inner sum over $r$ can be rewritten as 
\begin{align*}
    -p^{(k_{2}+1)(2m-s)-1}+\sum_{r=k_{2}+1}^{k_{2}+k^{\prime}-1}p^{r(2m-s)}(1-p^{2m-s-1})+p^{(k_{2}+k^{\prime})(2m-s)}.
\end{align*}
Hence $\sum\limits_{r=k_{2}+1}^{k_{2}+k^{\prime}}p^{-rs}C_{r,\eta}$ is equal to 
\begin{align*}
    \left(p^{(k_{2}+k^{\prime})(2m-s)+m}-p^{(k_{2}+1)(2m-s)+m-1}\right)\sum_{j=0}^{k_{1}}p^{j(2m-1)}+(1-p^{2m-s-1})R_{2}(p^{2m-s}),
\end{align*}
where 
\begin{align*}
    R_{2}(X)=\sum_{r=k_{2}+1}^{k_{2}+k^{\prime}-1}p^{m}\left(\sum_{j=0}^{k_{1}}p^{j(2m-1)}\right)X^{r}\in \Z[X].
\end{align*}

For the sum from $k_{2}+k^{\prime}+1$ to $k+1$,
we have 
\begin{align*}
    &\sum_{r=k_{2}+k^{\prime}+1}^{k+1}p^{-rs}C_{r,\eta}\\
    =&\sum_{r=k_{2}+k^{\prime}+1}^{k+1}\left(\sum_{j=0}^{k-r}\left(p^{r(2m-s)+j(2m-1)+m}-p^{r(2m-s)+j(2m-1)+m-1}\right)-p^{(2k+3)m+r(1-s)-k-2}\right)\\
    =&\sum_{j=0}^{k_{1}-1}p^{j(2m-1)+m}\left(\sum_{r=k_{2}+k^{\prime}+1}^{k-j}\left(p^{r(2m-s)}-p^{r(2m-s)-1}\right)-p^{(k-j+1)(2m-s)-1}\right)\\
    &-p^{(k_{2}+k^{\prime}+1)(2m-s)+k_{1}(2m-1)+m-1}.
\end{align*}
Use a similar trick to the other two parts,
we obtain that:
\begin{align*}
   \sum_{r=k_{2}+k^{\prime}+1}^{k+1}p^{-rs}C_{r,\eta}=-p^{(k_{2}+k^{\prime}+1)(2m-s)+m-1}\sum_{j=0}^{k_{1}}p^{j(2m-1)}+(1-p^{2m-s-1})R_{3}(p^{2m-s}),
\end{align*}
where 
\begin{align*}
    R_{3}(X)=\sum_{r=k_{2}+k^{\prime}+1}^{k}p^{m}\left(\sum_{j=0}^{k-r}p^{j(2m-1)}\right)X^{r}\in\Z[X].
\end{align*}
Putting all these $3$ parts together,
we have 
\begin{align*}
    \sum_{r=0}^{k+1}p^{-rs}C_{r,\eta}=&(1-p^{2m-s-1})(R_{1}+R_{2}+R_{3})(p^{2m-s})+(1-p^{3m-s-1})\sum_{r=0}^{k_{2}}p^{r(4m-s-1)}\\
    &+p^{k_{2}(2m-s)+m}\sum_{j=0}^{k_{2}}p^{j(2m-1)}+\left(p^{(k_{2}+k^{\prime})(2m-s)+m}-p^{(k_{2}+1)(2m-s)+m-1}\right)\sum_{j=0}^{k_{1}}p^{j(2m-1)}\\
    &-p^{(k_{2}+k^{\prime}+1)(2m-s)+m-1}\sum_{j=0}^{k_{1}}p^{j(2m-1)}\\
    =&(1-p^{2m-s-1})R_{\eta,p}(p^{2m-s})+\sum_{r=0}^{k_{2}}p^{r(4m-s-1)}-p^{3m-s-1}\sum_{r=0}^{k_{1}}p^{r(4m-s-1)},
\end{align*}
where 
\begin{align*}
    R_{\eta,p}(X)=R_{1}(X)+R_{2}(X)+R_{3}(X)+p^{m}\left(\sum_{j=0}^{k_{2}}p^{j(2m-1)}\right)X^{k_{2}}+p^{m}\left(\sum_{j=0}^{k_{1}}p^{j(2m-1)}\right)X^{k_{2}+k^{\prime}}
\end{align*}
is the desired polynomial $R_{\eta,p}(X)$ in the theorem.
The functional equation follows directly.
\end{proof}

\subsubsection{Split case}
\label{section split case of rank 2 coefficient}
When $p$ splits in $E$,
the self-dual Hermitian lattice $\bff{V}_{0}(\Z_{p})$ over 
$\oo_{E_{p}}\simeq\Z_{p}\times \Z_{p}$ 
inside the split Hermitian space $\bff{V}_{0}(\Q_{p})$ over $E_{p}\simeq \Q_{p}\times \Q_{p}$
can be identified with 
the direct sum of two copies of the quadratic $\Z_{p}$-lattice $L=\left(\Z_{p}^{n},((u_{i}),(v_{i}))=\sum_{i=1}^{n} u_{i}v_{i}\right)$, and one has 
\[\rmm{Tr}\lrangle{(x_{1},y_{1})}{(x_{2},y_{2})}=(x_{1},y_{2})+(x_{2},y_{1}),\text{ for any }x_{1},x_{2},y_{1},y_{2}\in L.\]
Taking $q(v)=\lrangle{v}{v}$,
we view $\bff{V}_{0}(\Z_{p})$ as a quadratic $\Z_{p}$-lattice,
and then we can apply \Cref{thm Siegel series orthogonal case}.
For any integers $r_{1},r_{2}\geq 0$, set 
\begin{align*}
        S_{T}(r_{1},r_{2}):=&\int_{v\in (p^{-r_1},p^{-r_2})\bff{V}_{0}(\Z_{p})} \chi_{T}^{-1}(v)\rmm{Char}(p^{\max(r_{1},r_{2})}\lrangle{v}{v}\in \Z_{p}) \, dv\\
        =&\int_{\substack{x\in p^{-r_1}L\\y\in p^{-r_2}L}} \psi_{p}^{-1}\left((x,T_{2})+(y,T_{1})\right)\rmm{Char}(p^{\max(r_{1},r_{2})}(x,y)\in \Z_{p}\, )dxdy.
\end{align*}
\begin{lemma}\label{lemma reduce split to orth}
    For any $T=(T_{1},T_{2})\in \bff{V}_{0}(\Q_{p})$ with $\lrangle{T}{T}\neq 0$ and $r_{1}\leq r_{2}$, 
    we have 
    \[S_{T}(r_{1},r_{2})=p^{n(r_{2}-r_{1})}B_{r_{1},(p^{r_{1}-r_{2}}T_{1},T_{2})},\]
    where $B_{r,\eta}$ is defined as in \Cref{thm Siegel series orthogonal case}.
\end{lemma} 
\begin{proof}
    Replace the variable $y\in p^{-r_{2}L}$ by $p^{r_{1}-r_{2}}y$, we obtain that 
    \begin{align*}
        S_{T}(r_{1},r_{2})&=\int_{\substack{x\in p^{-r_1}L\\ p^{r_1-r_2}y\in p^{-r_2} L}} \psi_{p}^{-1}\left((x,T_{2})+p^{r_{1}-r_{2}}(y,T_{1})\right)\rmm{Char}(p^{r_{2}}(x,p^{r_{1}-r_{2}}y)\in \Z_{p}) \, dx d(p^{r_{1}-r_{2}}y)\\
        &=p^{n(r_{2}-r_{1})}\int_{v=(x,y)\in p^{-r_1}\bff{V}_{0}(\Z_{p})} \psi_{p}^{-1}\left(\tr\lrangle{(p^{r_{1}-r_{2}}T_{1},T_{2})}{v}\right)\rmm{Char}(p^{r_{1}}q(v)\in \Z_{p})dv\\
        &=p^{n(r_{2}-r_{1})}B_{r_{1},(p^{r_{1}-r_{2}}T_{1},T_{2})}. \qedhere
    \end{align*}
\end{proof}
\begin{prop}\label{prop Siegel seris split case}
    Let $p$ be a prime that splits in $E$.
    For any $T=(T_1,T_2)\in \bff{V}_{0}(\Q_{p})$ with $\lrangle{T}{T}\neq 0$,
    $E_{2,p}^{T}(s)$ is non-zero if and only if $T\in \bff{V}_{0}(\Z_{p})$.
    In this case, if $p^{k}\|\lrangle{T}{T}$,
    then we have 
    \[E_{2,p}^T(s)=(1-p^{n-2s})Q_{T,p}(p^{\frac{n+1}{2}-s}),\]
    where $Q_{T,p}$ is a monic polynomial of degree $2k$ with coefficients in $\Z[p^{1/2}]$,
    satisfying the functional equation 
    \[X^{2k}Q_{T,p}(X^{-1})=Q_{T,p}(X).\]
    Moreover,
    the coefficients of even degree terms of $Q_{T,p}$ are all integral,
    and those of odd degree terms are all integral multiples of $p^{1/2}$.
\end{prop}
\begin{proof}
    Suppose that $p^{k_{1}}\|T_{1}$ and $p^{k_{2}}\| T_{2}$.
    We decompose the summation 
    \[E_{2,p}^{T}(s)=\sum_{r_{1},r_{2}\geq 0}p^{-(r_{1}+r_{2})s+\min(r_{1},r_{2})}S_{T}(r_{1},r_{2})\] 
    into 3 parts: $r_{1}=r_{2}$, $r_{1}<r_{2}$ and $r_{1}>r_{2}$.
    
    By \Cref{lemma reduce split to orth} and \Cref{thm Siegel series orthogonal case}, the $r_{1}=r_{2}$ part equals 
    \begin{align*}
        \sum_{r=0}^{\infty}p^{-r(2s-1)}B_{r,T}=(1-p^{n-2s})P_{T,p}(p^{n-2s+1}).
    \end{align*}
    Using \Cref{lemma reduce split to orth}, we rewrite the $r_{1}<r_{2}$ part as
    \begin{align*}
        \sum_{0\leq r_{1}<r_{2}}p^{-(r_{1}+r_{2})s+r_{1}}p^{n(r_{2}-r_{1})}B_{r_{1},(p^{r_{1}-r_{2}T_{1}},T_{2})}&=\sum_{r_{1}=0}^{\infty}\sum_{i=1}^{k_{1}}p^{-(2r_{1}+i)s+r_{1}+ni}B_{r_{1},(p^{-i}T_{1},T_{2})}\\
        &=\sum_{i=1}^{k_{1}}p^{i(n-s)}\left(\sum_{r_{1}=0}^{\infty}p^{-r_{1}(2s-1)}B_{r_{1},(p^{-i}T_{1},T_{2})}\right).
    \end{align*}
    The inner sum over $r_{1}$ equals $(1-p^{n-2s})P_{p,(p^{-i}T_{1},T_{2})}(p^{n-2s+1})$,
    thus the $r_{1}<r_{2}$ part equals 
    \[(1-p^{n-2s})\sum_{i=1}^{k_{1}}p^{\frac{i}{2}(n-1)}p^{i\left(\frac{n+1}{2}-s\right)}P_{(p^{-i}T_{1},T_{2}),p}(p^{n+1-2s}).\]
    By symmetry, the $r_{1}>r_{2}$ part equals 
    \[(1-p^{n-2s})\sum_{j=1}^{k_{2}}p^{\frac{j}{2}(n-1)}p^{j\left(\frac{n+1}{2}-s\right)}P_{(T_{1},p^{-j}T_{2}),p}(p^{n-2s+1}).\]
    Putting these $3$ parts together,
    we have 
    \[E_{2,p}^{T}(s)=(1-p^{n-2s})Q_{T,p}(p^{\frac{n+1}{2}-s}),\]
    where 
    \[Q_{T,p}(X)=P_{T,p}(X^{2})+\sum_{i=1}^{k_{1}}p^{\frac{i}{2}(n-1)}X^{i}P_{(p^{-i}T_{1},T_{2}),p}(X^{2})+\sum_{j=1}^{k_{2}}p^{\frac{j}{2}(n-1)}X^{j}P_{(T_{1},p^{-j}T_{2}),p}(X^{2}).\]
    By \Cref{thm Siegel series orthogonal case}, 
    $Q_{T,p}$ is zero when $T\notin \bff{V}_{0}(\Z_{p})$,
    and is a monic polynomial of degree $2k$ with coefficients in $\Z[p^{1/2}]$ when $T\in \bff{V}_{0}(\Z_{p})$.
    The functional equation follows from \eqref{eqn functional eqn split case}.
\end{proof}

\begin{cor}\label{cor unramified Siegel series split case} 
    Let $p$ be a prime that splits in $E$.
    For any $T\in\bff{V}_{0}(\Z_{p})$ with $\lrangle{T}{T}\in \Z_{p}^{\times}$,
we have 
$$E_{2,p}^T(s)=1-p^{n-2s}.$$
\end{cor}

\subsubsection{Inert case}

When $p$ is inert in $E$, one can also view the self-dual Hermitian lattice $\bff{V}_0(\Z_{p})$
as a split rank $2n$ quadratic lattice over $\Z_p$,
equipped with the quadratic form $q(v)=\lrangle{v}{v}$.
Then 
$$E_{2,p}^T(s)=\sum_{r=0}^\infty p^{-r(2s-1)}B_{r,T},\text{ where }B_{r,T}=\int_{v\in p^{-r}\bff{V}_{0}(\Z_{p})}\chi_{T}^{-1}(v)\rmm{Char}\left(p^{r}\lrangle{v}{v}\in \Z_{p}\right)dv.$$
Applying \Cref{thm Siegel series orthogonal case},
we obtain the following result:
\begin{prop}
    \label{prop Siegel series inert case}
    Let $p$ be a prime inert in $E$.
    For any $T\in\bff{V}_{0}(\Q_{p})$ with $\lrangle{T}{T}\neq 0$,
    $E_{2,p}^{T}(s)$ is non-zero if and only if $T\in\bff{V}_{0}(\Z_{p})$.
    In this case, 
    if $p^{k}\|\lrangle{T}{T}$,
    then we have 
    \[E_{2,p}^{T}(s)=(1-p^{n-2s})Q_{T,p}^{\prime}(p^{n-2s+1}),\]
    where 
    $Q_{T,p}^{\prime}$ is a monic polynomial of degree $k$ in $\Z[X]$,
    satisfying the functional equation 
    \[X^{k}Q_{T,p}^{\prime}(X^{-1})=Q_{T,p}^{\prime}(X).\]
\end{prop}
\begin{cor}\label{cor unramified Siegel series inert case}
    Let $p$ be a prime inert in $E$.
    For any $T\in \bff{V}_{0}(\Z_{p})$ with $\lrangle{T}{T}\in \Z_{p}^{\times}$,
    we have 
    \[E_{2,p}^{T}(s)=1-p^{n-2s}.\]
\end{cor}
To be coherent with other cases,
let $Q_{T,p}(X)$ be $Q_{T,p}^{\prime}(X^{2})$,
which is a monic even polynomial of degree $2k$ in $\Z[X]$,
satisfying the functional equation 
\[X^{2k}Q_{T,p}(X^{-1})=Q_{T,p}(X),\]
and 
we have 
\[E_{2,p}^{T}(s)=(1-p^{n-2s})Q_{T,p}(p^{\frac{n+1}{2}-s}).\]

\subsubsection{Ramified case}
For $p\neq 2$ ramified in $E$,
again we view $\bff{V}_{0}(\Z_{p})$,
together with $q(v)=\lrangle{v}{v}$,
as a quadratic $\Z_{p}$-lattice in $\bff{V}_{0}(\Q_{p})$.
However, in this case, it is no longer isometric to the quadratic $\Z_{p}$-lattice given in \Cref{thm Siegel series orthogonal case}.

Assume that $n=2m$ is even, and 
let $e_{1},\ldots,e_{m},f_{1},\ldots,f_{m}$ be a hyperbolic basis for the self-dual Hermitian lattice $\bff{V}_{0}(\Z_{p})$ in the split Hermitian space $\bff{V}_{0}(\Q_{p})$,
\emph{i.e.\,}
\[\lrangle{e_{i}}{e_{j}}=0,\,\lrangle{f_{i}}{f_{j}}=0,\,\lrangle{e_{i}}{f_{j}}=\delta_{i,j},\text{ for any }1\leq i,j\leq m.\]
Let $\varpi$ be a uniformizer of $\oo_{E_{p}}$ such that 
$\tr(\varpi)=0$ and $\varpi^{2}=pu$ for some unit $u\in \Z_{p}^{\times}$.
For any vector $v\in \bff{V}_{0}(\Z_{p})$,
if we write it as 
\[v=\sum_{i=1}^{m}\left((x_{i}+x_{i+m}\varpi)e_{i}+(y_{i}+y_{i+m}\varpi)f_{i}\right),\,x_{i},y_{i}\in \Z_{p},\text{ for any }1\leq i\leq 2m,\]
then 
\begin{align*}
    \lrangle{v}{v}=\sum_{i=1}^{m}\tr((x_{i}+x_{i+m}\varpi)\cdot(y_{i}-y_{i+m}\varpi))=2\sum_{i=1}^{m}\left(x_{i}y_{i}-pux_{i+m}y_{i+m}\right).
\end{align*}
Hence,
the $\Z_{p}$-lattice $\bff{V}_{0}(\Z_{p})$ is isometric to the one in \Cref{thm Siegel series ramified case}.
\begin{prop}\label{prop Siegel series unitary ramified case}
    Let $p$ be an odd prime ramified in $E$.
    For any $T\in\bff{V}_{0}(\Q_{p})$ with $\lrangle{T}{T}\neq 0$,
    $E_{2,p}^{T}(s)$ is non-zero if and only if $T\in \bff{V}_{0}(\Z_{p})$.
    In this case, if $p^{k}\|\lrangle{T}{T}$,
    then we have 
    \[E_{2,p}^{T}(s)=(1-p^{\frac{n}{2}-s})Q_{T,p}(p^{\frac{n+1}{2}-s}),\]
    where $Q_{T,p}$ is a monic polynomial of degree $2k$ with coefficients in $\Z[p^{1/2}]$, satisfying the functional equation:
    \[X^{2k}Q_{T,p}(X^{-1})=Q_{T,p}(X).\]
     Moreover,
    the coefficients of even degree terms of $Q_{T,p}$ are all integral,
    and those of odd degree terms are all integral multiples of $p^{1/2}$.
\end{prop}
\begin{proof}
    We divide the Siegel series $E_{2,p}^{T}(s)$ into two parts according to the parity of $r$:
\begin{align*}
    E_{2,p}^{T}(s)=&\sum_{r\geq 0}p^{-2rs+r}\int_{v\in p^{-r}\bff{V}_{0}(\Z_{p})}\psi_{p}^{-1}\left(\left(T/\varpi,v\right)\right)\mathrm{Char}(p^{r}\lrangle{v}{v}\in \Z_{p})dv\\
    &+\sum_{r\geq 1}p^{-(2r-1)s+r}\int_{v\in \varpi p^{-r}\bff{V}_{0}(\Z_{p})}\psi_{p}^{-1}\left(\left(T,v/\overline{\varpi}\right)\right)\mathrm{Char}(p^{r-1}\lrangle{v}{v}\in\Z_{p})dv\\
    =&\sum_{r\geq 0}p^{-r(2s-1)}\int_{v\in p^{-r}\bff{V}_{0}(\Z_{p})}\psi_{p}^{-1}\left(\left(T/\varpi,v\right)\right)\rmm{Char}(p^{r}\lrangle{v}{v}\in\Z_{p})dv\\
    &+p^{s-n}\sum_{r\geq 1}p^{-r(2s-1)}\int_{v\in p^{-r}\bff{V}_{0}(\Z_{p})}\psi_{p}^{-1}\left(\left(T,v\right)\right)\rmm{Char}(p^{r}\lrangle{v}{v}\in \Z_{p})dv\\
    =&\sum_{r\geq 0}p^{-r(2s-1)}C_{r,T/\varpi}+p^{s-n}\left(-1+\sum_{r\geq 0}p^{-r(2s-1)}C_{r,T}\right).
\end{align*}
    Set $k_{1}=v_{p}(T)$, $k_{2}=v_{p}(\varpi T)$, $k=v_{p}(\lrangle{T}{T})$ and $k^{\prime}=k-k_{1}-k_{2}$.
    These numbers are exactly those associated to $T$ in the setting of \Cref{thm Siegel series ramified case}.
    Applying \Cref{thm Siegel series ramified case},
    we obtain that:
    \begin{align*}
        E_{2,p}^{T}(s)=&(1-p^{2m-2s})R_{T/\varpi,p}(p^{2m-2s+1})+\sum_{r=0}^{k_{1}}p^{r(4m-2s)}-p^{3m-2s}\sum_{r=0}^{k_{2}-1}p^{r(4m-2s)}\\
        &+p^{s-2m}\left((1-p^{2m-2s})R_{T,p}(p^{2m-2s+1})+\sum_{r=1}^{k_{2}}p^{r(4m-2s)}-p^{3m-2s}\sum_{r=0}^{k_{1}}p^{r(4m-2s)}\right)\\
        =&(1-p^{2m-2s})\left(R_{T/\varpi,p}(p^{2m-2s+1})+p^{s-2m}R_{T,p}(p^{2m-2s+1})\right)\\
        &+(1-p^{m-s})\left(\sum_{r=0}^{k_{1}}p^{r(4m-2s)}+p^{2m-s}\sum_{r=0}^{k_{2}-1}p^{r(4m-2s)}\right)\\
        =&(1-p^{m-s})Q_{T,p}(p^{m-s+\frac{1}{2}}),
    \end{align*}
    where 
    \begin{align*}
        Q_{T,p}(X)=&(1+p^{-1/2}X)\left(R_{T/\varpi}(X^{2})+p^{-m+\frac{1}{2}}X^{-1}R_{T,p}(X^{2})\right)\\
        &+\sum_{r=0}^{k_{1}}p^{r(2m-1)}X^{2r}+\sum_{r=0}^{k_{2}-1}p^{r(2m-1)+m-\frac{1}{2}}X^{2r+1}\\
        =& R_{T/\varpi,p}(X^{2}) + p^{-\frac{1}{2}}X\left(R_{T/\varpi,p}(X^{2})+\sum_{r=0}^{k_{2}-1}p^{r(2m-1)+m}X^{2r}\right)\\
        &+p^{-m+\frac{1}{2}}X^{-1}R_{T,p}(X^{2})+\left(p^{-m}R_{T,p}(X^{2})+\sum_{r=0}^{k_{1}}p^{r(2m-1)}X^{2r}\right).
    \end{align*}
    By \eqref{eqn functional equation ramified case}, the degree $2k-2$ polynomial $R_{T,\varpi}(X^{2})$ satisfies 
    \[R_{T/\varpi,p}(X^{-2})=X^{-2(k-1+1)}R_{T/\varpi,p}(X^{2})=p^{-2k}R_{T/\varpi,p}(X^{2}),\]
    and the degree $2k-1$ polynomial $X^{-1}R_{T,p}(X^{2})$ satisfies
    \[(X^{-1})^{-1}R_{T,p}(X^{-2})=X^{-1}\cdot X^{2}\cdot X^{-2(k+1)}R_{T,p}(X^{2})=X^{-2k}\cdot X^{-1}R_{T,p}(X^{2}).\]
    On the other hand, 
    set 
    \[Q_{1}(X):=p^{-m}X R_{T/\varpi,p}(X^{2})+\sum_{r=0}^{k_{2}-1}p^{r(2m-1)}X^{2r+1}.\]
    The polynomial
    \begin{align*}
        Q_{1}(X)=\sum_{r=0}^{k_{2}-1}\frac{p^{(r+1)(2m-1)}-1}{p^{2m-1}-1}X^{2r+1}+\frac{p^{k_{2}(2m-1)}-1}{p^{2m-1}-1}\sum_{r=k_{2}}^{k_{1}+k^{\prime}-1}X^{2r+1}+\sum_{r=k_{1}+k^{\prime}}^{k-1}\frac{p^{(k-r)(2m-1)-1}}{p^{2m-1}-1}X^{2r+1}
    \end{align*}
    has degree $2k-1$ and satisfies 
    $Q_{1}(X^{-1})=X^{-2k}Q_{1}(X)$.    
    
    Similarly,
    set 
    \[Q_{2}(X):=p^{-m}R_{T,p}(X^{2})+\sum_{r=0}^{k_{1}}p^{r(2m-1)}X^{2r}.\]
    The polynomial 
    \[Q_{2}(X)=\sum_{r=0}^{k_{1}}\frac{p^{(r+1)(2m-1)}-1}{p^{2m-1}-1}X^{2r}+\frac{p^{(k_{1}+1)(2m-1)}-1}{p^{2m-1}}\sum_{r=k_{1}+1}^{k_{2}+k^{\prime}-1}X^{2r}+\sum_{r=k_{2}+k^{\prime}}^{k}\frac{p^{(k-r+1)(2m-1)}-1}{p^{2m-1}-1}X^{2r}\]
    has degree $2k$ and satisfies $Q_{2}(X^{-1})=X^{-2k}Q_{2}(X)$.
    In conclusion, $Q_{T,p}(X)$ satisfies all the desired properties.
\end{proof}
\begin{cor}\label{cor unramified Siegel series ramified case}
    Let $p$ be an odd prime ramified in $E$.
    For any $T\in\bff{V}_{0}(\Z_{p})$ with $\lrangle{T}{T}\in\Z_{p}^{\times}$,
    we have 
    \[E_{2,p}^{T}(s)=1-p^{\frac{n}{2}-s}.\]
\end{cor}

\subsection{Rank 2 Fourier coefficients: archimedean components}
\label{section archimedean place rank 2 Fourier infinite part}
For the archimedean part $E_{2,\infty}^{T}$,
we first analyze the function $f_{\ell,\infty}(w_{2}n,s=\ell+1)$ on $n\in\bff{N}(\R)$.
\begin{lemma}
    \label{lemma rank 2 values of archimedean section}
    For any $v\in V_{0}=\bff{V}_{0}(\R)$ and $x\in \R$,
    set:
    \[\alpha:=\alpha(v,x)=-\frac{\lrangle{v}{v}}{2}+ix+1,\text{ and }\beta:=\beta(v)=\sqrt{2}\lrangle{v}{u_{2}}.\]
    Then we have 
    \[\pi^{2\ell+1}f_{\ell,\infty}(w_{2}n(v,ix),s=\ell+1)=\frac{(\alpha u_{1}+\beta u_{2})^{\ell}(-\overline{\beta}u_{1}+\overline{\alpha}u_{2})^{\ell}}{(|\alpha(v,x)|^{2}+|\beta(v)|^{2})^{\ell+s}(\ell!)^{2}}.\]
\end{lemma}
\begin{proof}
    From the explicit action of $\bff{P}$ on $\bff{V}$, 
    one has 
    \[b_{1}w_{2}n(v,ix)=(-\lrangle{v}{v}/2+ix)b_{1}+b_{2}+v.\]
    Under the Iwasawa decomposition $\bff{G}(\R)=\bff{P}(\R)K_{\infty}$, 
    we write $w_{2}n(v,ix)$ as $pk$ for some $p\in \bff{P}(\R)$ and $k\in K_{\infty}$.
    Then 
    \begin{equation}\label{eqn searching Iwasawa decomposition}
         b_{1}w_{2}n(v,ix)=\nu(p)^{-1}b_{1}k.
    \end{equation}  
    Let $k_{+}$ be the factor of $k$ in $\rmm{U}(V_{2}^{+})$.
    Taking the $V_{2}^{+}$ components of both sides of \eqref{eqn searching Iwasawa decomposition},
    we obtain:
    \[\frac{1}{\sqrt{2}}\left(\alpha(v,x)u_{1}+\beta(v)u_{2}\right)=\nu(p)^{-1}\frac{u_{1}k_{+}}{\sqrt{2}}.\]
    The norms of both sides give us the identity $|\alpha(v,x)|^{2}+|\beta(v)|^{2}=|\nu(p)|^{-2}$.
    One may assume that 
    $\nu(p)^{-1}=\sqrt{|\alpha(v,x)|^{2}+|\beta(v)|^{2}}$,
    then with respect to the basis $\{u_{1},u_{2}\}$,
    the element $k^{+}\in \rmm{U}(V_{2}^{+})$ can be written as the Hermitian matrix 
    \[\frac{1}{\sqrt{|\alpha(v,x)|^{2}+|\beta(v)|^{2}}}\left(\begin{matrix}
        \alpha(v,x) & -\overline{\beta(v)}\\
        \beta(v) & \overline{\alpha(v,x)}
    \end{matrix}\right).\]
    Plugging these into $\pi^{2\ell+1}f_{\ell,\infty}(w_{2}n(v,ix),\ell+1)=|\nu(p)|^{2s}[u_{1}^{\ell}][u_{2}^{\ell}]\cdot k^{+}$,
    we get the desired value.
\end{proof}
Using \Cref{lemma rank 2 values of archimedean section},
one can write the archimedean coefficient 
$E_{2,\infty}^{T}:=E_{2,\infty}^{T}(1,\ell+1)$ 
as a concrete integral over $v\in V_{0}$ and $x\in\R$.
In particular,
we have the following formula for the coefficient of $[u_{1}^{\ell}][u_{2}^{\ell}]$ of $E_{2,\infty}^{T}$, denoted by $I_{0}(T;\ell)$:
\begin{align*}
    I_{0}(T;\ell)=&\pi^{-2\ell-1}\int_{v\in V_{0}}\int_{x\in \R}\chi_{T}^{-1}(v)\sum_{k=0}^{\ell}\frac{{\ell \choose k}^{2}(x^{2}+A)^{k}(-B)^{\ell-k}}{(x^{2}+A+B)^{2\ell+1}}dxdv\\
    =&\pi^{-2\ell-1}\int_{v\in V_{0}}\chi_{T}^{-1}(v)\sum_{k=0}^{\ell}(-B)^{\ell-k}{\ell\choose k}^{2}\int_{x\in \R}\frac{(x^{2}+A)^{k}}{(x^{2}+A+B)^{2\ell+1}}dxdv
\end{align*}
where $A=A(v):=|\alpha|^{2}-x^{2}=\left(1-\frac{\lrangle{v}{v}}{2}\right)^{2}$, and $B=B(v):=|\beta|^{2}=2|\lrangle{v}{u_{2}}|^{2}$.

By \cite[Theorem 3.5]{HMY}, for a non-zero vector $T\in V_{0}$,
when $\lrangle{T}{T}<0$,
the $\mbb{V}_{\ell}$-valued function $E_{2,\infty}^{T}(g,\ell+1)$ is identically zero on $\bff{G}(\R)$;
when $\lrangle{T}{T}\geq 0$, 
$E_{2,\infty}^{T}(g,\ell+1)$ must be a multiple of the generalized Whittaker function $\call{W}_{T}(g)$ given by \eqref{eqn generalized Whittaker function}.
To know this multiple exactly,
it suffices to compare the coefficients of $[u_{1}^{\ell}][u_{2}^{\ell}]$ in $E_{2,\infty}^{T}$ and $\call{W}_{T}(1)$,
\emph{i.e.\,}compare $I_{0}(T;\ell)$ and 
$K_{0}(4\sqrt{2}\pi |\lrangle{u_{2}}{T}|)=K_{0}(4\pi\sqrt{B(T)})$.
\subsubsection{Fourier transforms}
\label{section Fourier transform of coefficients}
Equipped with the bilinear form $(v,w):=\tr_{\C}\lrangle{v}{w}=2\operatorname{Re} \lrangle{v}{w}$,
we view $V_{0}$ as a quadratic space with signature $(2,2n-2)$,
and define the Fourier transform of a Schwartz function $f$ on $V_{0}$ to be 
\[\widehat{f}(T):=\int_{v\in V_{0}}f(v)e^{-2\pi i (T,v)}dv,\,T\in V_{0}.\]

We use the following lemma to write $I_{0}(T;\ell)$ as a Fourier transform on $V_{0}$:
\begin{lemma}
    \label{lemma integral over center of Heisenberg}
    For any real number $C,D$
    and two natural numbers $m<n$,
    we have 
    \[\int_{\R}\frac{(x^{2}+C)^{m}}{(x^{2}+D)^{n}}dx=\frac{D^{m-n+\frac 12}}{(n-1)!}\sum_{k=0}^{m}{m\choose k}\left(\frac{C}{D}\right)^{m-k}\Gamma(k+\tfrac 12)\Gamma(n-k-\tfrac 12).\]
\end{lemma}
\begin{proof}
    An exercise of calculus.
\end{proof}
Applying \Cref{lemma integral over center of Heisenberg} to $I_{0}(T;\ell)$, 
we obtain that $I_{0}(T;\ell)=\widehat{F_{0,\ell}}(T)$, where 
\begin{align*}
    F_{0,\ell}(v):=\pi^{-2\ell}\sum_{k=0}^{\ell}(-B)^{\ell-k}{\ell\choose k}^{2}\frac{(A+B)^{k-2\ell-\frac 12}}{(2\ell)!}\sum_{j=0}^{k}{k\choose j}\left(\frac{A}{A+B}\right)^{k-j}\frac{\cdot(2j)!(4\ell-2j)!}{2^{4\ell}\cdot j!(2\ell-j)!}.
\end{align*}
Now we can reduce the problem to the comparison of $F_{0,\ell}(T)$
and the inverse Fourier transform of (the product of a certain function with) $K_{0}(4\pi\sqrt{B(T)})$.
Recall the following result for quadratic spaces by Pollack:
\begin{prop}\label{prop Fourier transform of K Bessel functions}
    ($v=0$ case of \cite[Corollary 4.5.4]{Po})
    Let $(V^{\prime},q)=V_{2}\oplus V_{m}$ be a non-degenerate quadratic space over $\R$ of signature $(2,m)$,
    with associated symmetric bilinear form $B_{q}(v,w)=q(v+w)-q(v)-q(w)$,
    such that $q|_{V_{2}}$ is positive definite and $q|_{V_{m}}$ is negative definite.
    Let $v_{1}^{\prime},v_{2}^{\prime}$ be an orthonormal basis of $V_{2}$.
    For $x=x_{2}+x_{m}\in V_{2}\oplus V_{m}$,
    define $\|x\|^{2}:=2q(x_{2})-2q(x_{m})$.
    For any integer $\ell>m/2$ and $x=x_{2}+x_{m}\in V_{2}\oplus V_{m}$ with $\|x_{2}\|+\|x_{m}\|<\sqrt{2}$,
    we have 
    \begin{align*}
        &\int_{V^{\prime}}e^{i B_{q}(\omega,x)}\rmm{Char}(q(\omega)>0)q(\omega)^{\ell-m/2}K_{0}\left(\sqrt{2\left(B_{q}(\omega,v_{1}^{\prime})^{2}+B_{q}(\omega,v_{2}^{\prime})^{2}\right)}\right)d\omega\\
        =&2^{\ell}\pi^{\frac{m+2}{2}}\ell!\cdot\Gamma(\ell-m/2+1)\tau(\sqrt{2} x_{m},\sqrt{2}x_{2})^{-\ell-1}{_{2}F_{1}}\left(-\frac{\ell}{2},\frac{\ell+1}{2};1;\frac{2\|x_{2}\|^{2}}{\tau(\sqrt{2}x_{m},\sqrt{2}x_{2})^{2}}\right),
    \end{align*}
    where $\tau(\sqrt{2}x_{m},\sqrt{2}x_{2})^{2}:=\left(1+\frac{\|x_{m}\|^{2}-\|x_{2}\|^{2}}{2}\right)^{2}+2\|x_{2}\|^{2}$.
    Here \[{_{2}F_{1}}(a,b;c;z)=\sum_{n=0}^{\infty}\frac{(a)_{n}(b)_{n}}{(c)_{n}}\cdot\frac{z^{n}}{n!},\,|z|<1\] 
    is the hypergeometric function, where $(q)_{n}=\frac{\Gamma(q+n)}{\Gamma(q)}$ is the (rising) Pochhammer symbol.
\end{prop}

Take $(V^{\prime},q)$ to be $(V_{0},q(\omega)=\frac{\lrangle{\omega}{\omega}}{2})$,
as a quadratic space of signature $(2,2n-2)$,
and $v^{\prime}_{1}=u_{2}$,
$v^{\prime}_{2}=i u_{2}$.
Now we rewrite the objects appearing in \Cref{prop Fourier transform of K Bessel functions} in terms of our unitary notations: 
\begin{itemize}
    \item Let $m=2n-2$.
    \item The K-Bessel function appearing in the integral is $K_{0}(\sqrt{2}|\lrangle{\omega}{u_{2}}|)=K_{0}(\sqrt{B(\omega)})$.
    \item The norm $\|x_{2}\|^{2}$ equals $|\lrangle{x}{u_{2}}|^{2}=B(x)/2$.
    \item The function $\tau(\sqrt{2}x_{m},\sqrt{2}x_{2})^{2}$ equals $(1-\frac{\lrangle{x}{x}}{2})^{2}+2|\lrangle{x}{u_{2}}|^{2}=A(x)+B(x)$.
\end{itemize}
Hence \Cref{prop Fourier transform of K Bessel functions} can be rewritten as:
\begin{align*}
    &2^{-\ell+n-1}\int_{V_{0}}e^{i(\omega,x)/2}\rmm{Char}(\lrangle{\omega}{\omega}>0)\lrangle{w}{w}^{\ell-n+1}K_{0}(\sqrt{B(\omega)})d\omega\\
    =&2^{\ell}\pi^{n}\ell!(\ell-n+1)!(A(x)+B(x))^{-\frac{\ell+1}{2}}{_{2}F_{1}}\left(-\frac{\ell}{2},\frac{\ell+1}{2};1;\frac{B(x)}{A(x)+B(x)}\right).
\end{align*}
We replace the variable $\omega$ in the integral by $4\pi \omega$ and 
replace $\ell$ by $2\ell$ with $\ell>\frac {n-1}2$, then the inverse Fourier transform of 
$\rmm{Char}(\lrangle{\omega}{\omega}>0)\lrangle{\omega}{\omega}^{2\ell-n+1}K_{0}(4\pi \sqrt{B(\omega)})$
is 
\begin{align}\label{eqn inverse Fourier of generalized Whittaker}
       2^{-4\ell-n-3} \pi^{-4\ell+n-2}(2\ell)!(2\ell-n+1)!(A(x)+B(x))^{-\ell-\frac 12}{_{2}F_{1}}\left(-\ell,\ell+\tfrac 12;1;\tfrac{B(x)}{A(x)+B(x)}\right).
\end{align}

\subsubsection{Write \texorpdfstring{$F_{0,\ell}$}{} in terms of hypergeometric functions}
\label{section write F as hypergeometric function}
Set $z:=z(v)=\frac {B(v)}{A(v)+B(v)}$. Then we have: 
\[F_{0,\ell}(v)=\frac{2^{-4\ell}}{(2\ell)!}\pi^{-2\ell}(A+B)^{-\ell-\frac 12}\sum_{k=0}^{\ell}\sum_{j=0}^{k}{\ell\choose k}^{2}{k\choose j}\frac{(2j)!(4\ell-2j)!}{j!(2\ell-j)!}(-z)^{\ell-k}(1-z)^{k-j}.\]
\begin{prop}\label{prop summation in terms of hypergeometric}
    For any integers $0\leq j\leq k\leq \ell$,
    set $\displaystyle c_{j,k}={\ell\choose k}^{2}{k\choose j}\frac{(2j)!(4\ell-2j)!}{j!(2\ell-j)!}.$
    Then 
    \begin{align}\label{eqn calculation of combinatorics in Fourier}
        \sum_{k=0}^{\ell}\sum_{j=0}^{k}c_{j,k}(-z)^{\ell-k}(1-z)^{k-j}=\sum_{r=0}^{\ell}(-1)^{r}2^{2\ell-2r}\frac{(2\ell)!(2\ell+2r)!}{(r!)^{2}(\ell+r)!(\ell-r)!}z^{r}.
    \end{align} 
\end{prop}
\begin{proof}
    We write:
\[\sum_{k=0}^{\ell}\sum_{j=0}^{k}c_{j,k}(-z)^{\ell-k}(1-z)^{k-j}=\sum_{r=0}^{\ell}(-1)^{r}C(r)z^{r}.\]
    The product $(-z)^{\ell-k}(1-z)^{k-j}$ has a non-zero $z^{r}$ term if and only if 
$\ell-k\leq r\leq \ell-j$. Thus
\[(-1)^{r}C(r)=\sum_{j=0}^{\ell-r}\sum_{k=\ell-r}^{\ell}c_{j,k}{{k-j}\choose {r-\ell+k}}(-1)^{\ell-k+r-\ell+k}=(-1)^{r}\sum_{j=0}^{\ell-r}\sum_{k=\ell-r}^{\ell}c_{j,k}{{k-j}\choose {r-\ell+k}}.\]

To compute $C(r)$,
we need the following lemma:
\begin{lemma}
    \label{lemma combinatorics summations}
    \begin{enumerate}
        \item For integers $0\leq a\leq b$, one has \[\sum_{i=a}^{b}{b\choose i}{{b-a}\choose{b-i}}={{2b-a}\choose b}.\]
        \item For any integer $0\leq r\leq \ell$, one has 
        \[\sum_{i=0}^{\ell-r}\frac{\binom{2\ell}{i}\binom{\ell-r}{i}}{\binom{4\ell}{2i}}=2^{2\ell-2r}\frac{\binom{2\ell+2r}{\ell+r}}{\binom{4\ell}{2\ell}}.\]
    \end{enumerate}
\end{lemma}
\begin{proof}
    The identity in (1) is obvious.
    For the identity in (2),
    rewrite the left-hand side as:
    \begin{align*}
        \sum_{i=0}^{\ell-r}\frac{(-\ell+r)_{i}(\tfrac 12)_{i}}{(-2\ell+\tfrac 12)_{i}}\frac{1}{i!}={_{2}F_{1}}(-(\ell-r),1/2;-2\ell+1/2;1).
    \end{align*}
     By Chu-Vandermond identity,
    this value of hypergeometric function equals
    \[\frac{(-2\ell)_{\ell-r}}{(-2\ell+\tfrac 12)_{\ell-r}}
    =2^{2\ell-2r}\frac{\binom{2\ell+2r}{\ell+r}}{\binom{4\ell}{2\ell}}.\qedhere\]
\end{proof}
Now we return to the calculation of $C(r)$, applying \Cref{lemma combinatorics summations}:
\begin{align*}
    C(r)&=\sum_{j=0}^{\ell-r}\sum_{k=\ell-r}^{\ell}\frac{(\ell!)^{2}(2j)!(4\ell-2j)!}{k!((\ell-k)!)^{2}(j!)^{2}(r-\ell+k)!(\ell-r-j)!(2\ell-j)!}\\
    &=\frac{\ell!}{r!}\sum_{j=0}^{\ell-r}\frac{(2j)!(4\ell-2j)!}{(j!)^{2}(\ell-r-j)!(2\ell-j)!}\sum_{k=\ell-r}^{\ell}\binom{\ell}{k}\binom{r}{\ell-k}\\
    &=\frac{\ell!}{r!}\cdot\frac{(\ell+r)!}{\ell!r!}\sum_{j=0}^{\ell-r}\frac{(2j)!(4\ell-2j)!}{(j!)^{2}(\ell-r-j)!(2\ell-j)!}\\
    &=\frac{(\ell+r)!}{(r!)^{2}}\cdot \frac{(4\ell)!}{(\ell-r)!(2\ell)!}\sum_{j=0}^{\ell-r}\frac{\binom{2\ell}{j}\binom{\ell-r}{j}}{\binom{4\ell}{2j}}\\
    &=\frac{(\ell+r)!}{(r!)^{2}}\cdot\frac{(4\ell)!}{(\ell-r)!(2\ell)!}\cdot 2^{2\ell-2r}\frac{\frac{(2\ell+2r)!}{((\ell+r)!)^{2}}}{\frac{(4\ell)!}{((2\ell)!)^{2}}}\\
    &=2^{2\ell-2r}\frac{(2\ell)!(2\ell+2r)!}{(r!)^{2}(\ell+r)!(\ell-r)!}.\qedhere
\end{align*}
\end{proof}
Applying \Cref{prop summation in terms of hypergeometric} to $F_{0,\ell}$,
we have 
\begin{align*}
    F_{0,\ell}(v)
    =&\frac{2^{-4\ell}\pi^{-2\ell}}{(2\ell)!(A+B)^{\ell+\frac 12}}\sum_{r=0}^{\ell}(-1)^{r}2^{2\ell-2r}\frac{(2\ell)!(2\ell+2r)!}{(r!)^{2}(\ell+r)!(\ell-r)!}z^{r}\\
    =&\frac{2^{-2\ell}(2\ell)!\pi^{-2\ell}}{(\ell!)^{2}(A+B)^{\ell+\frac 12}}\sum_{r=0}^{\ell}\frac{(-\ell)_{r}(\ell+\frac 12)_{r}}{(1)_{r}}\frac{z^{r}}{r!}\\
    =&\frac{2^{-2\ell}(2\ell)!\pi^{-2\ell}}{(\ell!)^{2}(A+B)^{\ell+\frac 12}}\cdot {_{2}F_{1}}(-\ell,\ell+\tfrac 12;1;z).
\end{align*}
We have shown the following result:
\begin{prop}\label{prop archimedean rank 2 Fourier}
    For any $T\in V_{0}$,
    the coefficient $I_{0}(T;\ell)$ of $[u_{1}^{\ell}][u_{2}^{\ell}]$ in $E_{2,\infty}^{T}$ satisfies
    \[I_{0}(T;\ell)=\widehat{F_{0,\ell}}(T),\]
        where for $v\in V_{0}$, 
    \[F_{0,\ell}(v)=\frac{2^{-2\ell}(2\ell)!\pi^{-2\ell}}{(\ell!)^{2}(A+B)^{\ell+1/2}}\cdot {_{2}F_{1}}(-\ell,\ell+1/2;1;\tfrac{B}{A+B}),\]
    and $A(v)=\left(1-\frac{\lrangle{v}{v}}{2}\right)^{2},\,B(v)=2|\lrangle{v}{u_{2}}|^{2}$.
\end{prop}
Comparing \Cref{prop archimedean rank 2 Fourier} with \eqref{eqn inverse Fourier of generalized Whittaker},
the function 
\[\pi^{-2\ell+n-2}2^{-2\ell-n-3}(\ell!)^{2}(2\ell-n+1)!F_{0,\ell}(v)\]
is equal to the inverse Fourier transform of 
\[\rmm{Char}(\lrangle{\omega}{\omega}>0)\lrangle{\omega}{\omega}^{2\ell-n+1}K_{0}(4\pi \sqrt{B(\omega)}),\]
which implies the following theorem:
\begin{thm}\label{thm rank 2 Fourier archimedean generalised Whittaker}
    For any $T\in V_{0}$ and $\ell>n$,
    we have
    \[I_{0}(T;\ell)=\rmm{Char}(\lrangle{T}{T}>0)\frac{2^{2\ell+n+3}\pi^{2\ell-n+2}\lrangle{T}{T}^{2\ell-n+1}}{(\ell!)^{2}(2\ell-n+1)!}K_{0}(|\beta_{T}(1)|).\]
    In particular,
    $E_{2,\infty}^{T}(g,\ell+1)$ is identically zero when $T$ is non-zero and $\lrangle{T}{T}\leq 0$,
    and equals 
    \[\frac{2^{2\ell+n+3}\pi^{2\ell-n+2}\lrangle{T}{T}^{2\ell-n+1}}{(\ell!)^{2}(2\ell-n+1)!}\call{W}_{T}(g)\]
    when $\lrangle{T}{T}>0$.
\end{thm}
Combining all the local results \Cref{prop Siegel seris split case}, \Cref{prop Siegel series inert case}, \Cref{prop Siegel series unitary ramified case} and \Cref{thm rank 2 Fourier archimedean generalised Whittaker} together,
we obtain the rank $2$ Fourier coefficient of $E_{\ell}(g,s=\ell+1)$ at $g\in \bff{G}(\R)$:
\begin{thm}
    \label{thm rank 2 Fourier coefficient}
    Suppose $n\equiv 2 \modulo 4$, $\ell>n$,
    and $D\equiv 3\modulo 4$.
    For a vector $T\in \bff{V}_{0}$ with $\lrangle{T}{T}\neq 0$ and any $g_{\infty}\in \bff{G}(\R)$, 
    if $\lrangle{T}{T}<0$ or $T\notin \bff{V}_{0}(\Z)$,
    the $T$th Fourier coefficient $E^{T}_{\ell}(g_{\infty},s=\ell+1)$ is zero;
    if $T\in \bff{V}_{0}(\Z)$ and $\lrangle{T}{T}>0$,
    we have:
    \begin{align*}
        E^{T}_{\ell}(g_{\infty},\ell+1)=\frac{(2\ell-n+2)2^{2n+2}D^{\ell+1-\frac{n}{2}}}{(\ell!)^{2}\left|B_{2\ell-n+2}\right|\sigma_{\ell+1-\frac{n}{2}}(D)}\cdot\left(\prod_{p|\lrangle{T}{T}}Q_{T,p}(p^{\ell-\frac{n-1}{2}})\right)\call{W}_{T}(g_{\infty}),
    \end{align*}
    where $B_{k}$ is the $k$th Bernoulli number, and $\sigma_{k}(n)=\sum_{d|n}d^{k}$.
    
    In particular, for any vector $T\in\bff{V}_{0}$ with $\lrangle{T}{T}\neq 0$
    and any $g=g_{f}g_{\infty}\in \bff{M}(\A_{f})\times\bff{G}(\R)$,
    $E_{\ell}^{T}(g_{f}g_{\infty},\ell+1)$ is a rational multiple of $\call{W}_{T}(g_{\infty})$,
    and when $g_{f}=1$ the denominator of this rational number divides $(\ell!)^{2}\left|B_{2\ell-n+2}\right|\sigma_{\ell+1-\frac{n}{2}(D)}$,
    which is independent of $T$.
\end{thm}
\begin{proof}
    By our local results,
    we know that the $T$th coefficient of $E_{\ell}(g,s=\ell+1)$ is non-zero only if $T\in\bff{V}_{0}(\Z)$ and $\lrangle{T}{T}>0$,
    and for any $g_{\infty}\in \bff{M}(\R)$ we have:
    \begin{align*}
        E_{\ell}^{T}(g_{\infty},\ell+1)=&\frac{\prod\limits_{p}(1-p^{n-2\ell-2})Q_{T,p}(p^{\frac{n-1}{2}-\ell})}{\prod_{p|D}(1+p^{\frac{n}{2}-\ell-1})}\cdot \frac{2^{2\ell+n+3}\pi^{2\ell-n+2}\lrangle{T}{T}^{2\ell-n+1}}{(\ell!)^{2}(2\ell-n+1)!}\call{W}_{T}(g_{\infty}) \\
        =&\frac{\pi^{2\ell-n+2}}{\zeta(2\ell-n+2)}\cdot \frac{2^{2\ell+n+3}D^{\ell+1-\frac{n}{2}}}{(\ell!)^{2}(2\ell-n+1)!\sigma_{\ell+1-\frac{n}{2}}(D)}\cdot \prod_{p|\lrangle{T}{T}}\frac{Q_{T,p}(p^{\frac{n-1}{2}-\ell})}{p^{(n-1-2\ell)v_{p}(\lrangle{T}{T})}}.
    \end{align*}
    Since $n$ is even, the fraction $\pi^{2\ell-n+2}/\zeta(2\ell-n+2)$ is rational and equals 
    $\frac{(2\ell-n+2)!}{2^{2\ell-n+1}|B_{2\ell-n+2}|}$.
    By the functional equation for $Q_{T,p}$ and the properties of its coefficients,
    we have 
    \[\frac{Q_{T,p}(p^{\frac{n-1}{2}-\ell})}{p^{(n-1-2\ell)v_{p}(\lrangle{T}{T})}}=Q_{T,p}(p^{\ell-\frac{n-1}{2}})\in \Z.\]
    Putting these values into $E_{\ell}^{T}(g_{\infty},\ell+1)$, we obtain the desired formula.
\end{proof}


\section{Lower ranks Fourier coefficients of \texorpdfstring{$E_\ell(g, s=\ell+1)$}{}}
\label{section lower rank Fourier}
In this section, we explicitly compute the rank $0$ and $1$ Fourier coefficients of the weight $\ell$ quaternionic Heisenberg Eisenstein series $E_\ell(g, \ell+1)$,
\emph{i.e.\,} the $T$th coefficients for isotropic $T$.

\subsection{Rank 1 Fourier coefficients}
\label{section rank 1 Fourier coefficients}
For rank $1$ Fourier coefficients,
one needs to calculate 
\[E_{\ell}^{T}(g,s)=E_{\ell,1}^{T}(g,s)=\int_{\bff{N}_{T}(\A)\backslash \bff{N}(\A)}\chi_{T}^{-1}(n)f_{\ell}(\gamma(L_{T})ng,s)dn\]
for any nonzero isotropic vector $T$.

\subsubsection{Non-archimedean components}
\label{section rank 1 Fourier non-archimedean}
For any prime $p$,
we identify $\bff{N}_{T}(\Q_{p})\backslash \bff{N}(\Q_{p})$ with $E_{p}$ via the map $n(v,\lambda)\mapsto -\lrangle{T}{v}$.
For any $m=(z,h)\in \bff{M}(\Q_{p})$,
\begin{align*}
    E_{\ell,p}^{T}(m,s)&=\int_{x\in E_{p}}\int_{t\in E_{p}^{\times}}\psi_{E_{p}}(x)|t|_{E_{p}}^{s}\Phi_{p}(tx\nu(m)^{-1}b_{1}+tT\cdot h)dtdx\\
    &=|\nu(m)|_{E_{p}}^{s}\int_{x\in E_{p}}\int_{t\in E_{p}^{\times}}\psi_{E_{p}}(x)|t|_{E_{p}}^{s}\Phi_{p}(txb_{1}+tzT\cdot h)dtdx\\
    &=|\nu(m)|_{E_{p}}^{s}E_{\ell,p}^{zT\cdot h}(1,s),
\end{align*}
thus it suffices to study $E_{1,p}^{T}:=E_{\ell,p}^{T}(1,s=\ell+1)$.
We have:
\begin{align*}
    E_{1,p}^{T}&=\int_{\substack{x\in E_{p}\\ t\in E_{p}^{\times}}}\psi_{E_{p}}(x)|t|_{E_{p}}^{\ell+1}\Phi_{p}(txb_{1}+tT)dtdx=\int_{\substack{x\in E_{p}\\ t\in E_{p}^{\times}}}\psi_{E_{p}}(x/t)|t|_{E_{p}}^{\ell}\Phi_{p}(xb_{1}+tT)dtdx.
\end{align*}
\begin{defi}\label{def divisor function quadratic}
    For any vector $T\in\bff{V}_{0}$,
    define 
    $\sigma_{E,\ell}(T)=\prod_{\frakk{p}}\sigma_{\frakk{p},\ell}(T)$,
    where $\frakk{p}$ varies over all the non-zero prime ideals of $\oo_{E}$ and 
    \[\sigma_{\frakk{p},\ell}(T):=\sum_{i=0}^{v_{\frakk{p}}(T)}q^{i\ell},\text{ where }q=|\oo_{E}/\frakk{p}|.\]
    Here we take $\sigma_{\frakk{p},\ell}(T)$ to be $0$ if $v_{\frakk{p}}(T)<0$, \emph{i.e.\,}$T\notin\bff{V}_{0}(\Z_{p})$.
\end{defi}
\begin{prop}\label{prop non-arch rank 1 Fourier}
    For any non-zero isotropic vector $T\in\bff{V}_{0}$,
    we have 
    \[\prod_{p\text{ prime}}E_{1,p}^{T}=\sigma_{E,\ell}(T).\]
    In particular, $\prod_{p}E_{1,p}^{T}=0$ if $T\notin \bff{V}_{0}(\Z)$.
\end{prop}
\begin{proof}
    When $p=\frakk{p}_{1}\frakk{p}_{2}$ is split in $E$,
    we write $T$ as $(T_{1},T_{2})$ with $T_{i}\in \Q_{p}^{n}$,\,$i=1,2$
    such that $v_{p}(T_{i})=v_{\frakk{p}_{i}}(T)$ for $i=1,2$.
    We have:
    \begin{align*}
    E_{1,p}^{T}
    =&\sum_{\substack{r_{1}\geq -v_{\frakk{p}_{1}}(T)\\r_{2}\geq -v_{\frakk{p}_{2}}(T)}}p^{-\ell(r_{1}+r_{2})}\left(\int_{(x_{1},x_{2})\in \Z_{p}\times \Z_{p}}\psi_{p}(x_{1}/p^{r_{1}})\psi_{p}(x_{2}/p^{r_{2}})dx_{1}dx_{2}\right)\\
    =&\sum_{\substack{0\geq r_{1}\geq -v_{\frakk{p}_{1}}(T)\\0\geq r_{2}\geq -v_{\frakk{p}_{2}}(T)}}p^{-\ell(r_{1}+r_{2})}=\sigma_{\frakk{p}_{1},\ell}(T)\sigma_{\frakk{p}_{2},\ell}(T),
\end{align*}
where the second equality follows from the fact that the inner integral over $\Z_{p}\times\Z_{p}$ equals $1$ if and only if the character $(x_{1},x_{2})\mapsto \psi_{p}(x_{1}/p^{r_{1}})\psi_{p}(x_{2}/p^{r_{2}})$ is trivial,
otherwise it vanishes.

When $p$ is inert in $E$,
$E_{p}$ is an unramified quadratic field extension of $\Q_{p}$,
and we have:
\begin{align*}
    E_{1,p}^{T}=\sum_{r\geq -v_{p}(T)}p^{-2r \ell}\int_{x\in \oo_{E_{p}}}\psi_{E_{p}}(x/p^{r})dx=\sum_{0\geq r\geq -v_{p}(T)}p^{-2r\ell}=\sigma_{p,\ell}(T).
\end{align*}

When $p=\frakk{p}^{2}$ is ramified in $E$,
$E_{p}=E_{\frakk{p}}$ is a ramified quadratic field extension of $\Q_{p}$. 
Taking a uniformizer $\varpi$ of $\frakk{p}$, we have 
\begin{align*}
    E_{1,p}^{T}=\sum_{r\geq -v_{\frakk{p}}(T)}p^{-r \ell}\int_{x\in \oo_{E_{\frakk{p}}}}\psi_{E_{\frakk{p}}}(x/\varpi^{r})dx=\sum_{0\geq r\geq -v_{\frakk{p}}(T)}p^{-r\ell}=\sigma_{\frakk{p},\ell}(T).
\end{align*}
Here we use the property that the character $\psi_{E_{\frakk{p}}}$ has conductor $\oo_{E_{\frakk{p}}}$. 

Combining these three cases together, we obtain the desired result.
\end{proof}

\subsubsection{Archimedean components}
\label{section rank 1 Fourier archimedean}
Again, 
we identify $\bff{N}_{T}(\R)\backslash\bff{N}(\R)$ with $\C$ via the map $n(v,\lambda)\mapsto -\lrangle{T}{v}$,
and for any $z\in \C$, we choose an element $n(z)\in\bff{N}(\R)$ corresponding to $z$ under this identification. Recall from Lemma 3.8 that $L_T$ is the line generated by $T$.
\begin{lemma}
    \label{lemma rank 1 archimedean section}
    For any $z\in \C$,
    we have 
    \[\pi^{2\ell+1}f_{\ell,\infty}(\gamma(L_{T})n(z),s=\ell+1)=\frac{(zu_{1}+\beta u_{2})^{\ell}(-\overline{\beta}u_{1}+\overline{z}u_{2})^{\ell}}{(|z|^{2}+|\beta|^{2})^{2\ell+1}(\ell!)^{2}},\]
    where $\beta=\beta(T):=\sqrt{2}\lrangle{T}{u_{2}}$.
\end{lemma}
\begin{proof}
    Suppose that $\gamma(L_{T})n(z)=pk$ for $p\in \bff{P}(\R)$ and $k=(k_{+},k_{-})\in K_{\infty}=\rmm{U}(V_{2}^+)\times \rmm{U}(V_{n}^{-})$,
    then one has 
    \[T+zb_{1}=b_{1}\gamma(L_{T})n(z)=b_{1}pk=\nu(p)^{-1}b_{1}k.\]
    Taking the projection to $V_{2}^{+}$ of both sides,
    we obtain that 
    \[zu_{1}+\sqrt{2}\lrangle{T}{u_{2}}u_{2}=\nu(p)^{-1}u_{1}k_{+}.\]
    One can take $\nu(p)$ to be $(|z|^{2}+|\beta|^{2})^{-1/2}$,
    then 
    \[k_{+}=\nu(p)\mat{z}{-\overline{\beta}}{\beta}{\overline{z}}\in\rmm{U}(V_{2}^{+}),\]
    which gives us the desired value of $f_{\ell,\infty}$.
\end{proof}
Using \Cref{lemma rank 1 archimedean section},
the coefficient $I_{0}(T;\ell)$ of $[u_{1}^{\ell}][u_{2}^{\ell}]$ in $E_{1,\infty}^{T}:=E_{1,\infty}^{T}(1,s=\ell+1)$ is equal to 
\begin{align}
    \label{eqn central coefficient rank 1 archimedean}
    \pi^{-2\ell-1}\int_{z\in \C}e^{2\pi i \cdot 2 \mathrm{Re}(z)}\frac{\sum\limits_{k=0}^{\ell}\binom{\ell}{k}^{2}|z|^{2k}(-|\beta|^{2})^{\ell-k}}{(|z|^{2}+|\beta|^{2})^{2\ell+1}}dz.
\end{align}
We set $B=|\beta|^{2}$ and $z=x+iy$, where $x,y\in \R$,
then the integral \eqref{eqn central coefficient rank 1 archimedean} equals 
\begin{align}
    \label{eqn rank 1 central coefficient written as in rank 2}
   \pi^{-2\ell-1}\int_{\R}e^{4\pi i x}\sum_{k=0}^{\ell}\binom{\ell}{k}^{2}(-B)^{\ell-k}\left(\int_{\R}\frac{(y^{2}+x^{2})^{k}}{(y^{2}+x^{2}+B)^{2\ell+1}}dy \right)dx.
\end{align}
By Lemma \ref{lemma integral over center of Heisenberg},
the inner integral over $y\in\R$ equals 
\[\frac{(x^{2}+B)^{k-2\ell-1/2}}{(2\ell)!}\sum_{j=0}^{k}\binom{k}{j}\left(\frac{x^{2}}{x^{2}+B}\right)^{k-j}\Gamma(j+1/2)\Gamma(2\ell-j+1/2).\]
Applying \Cref{prop summation in terms of hypergeometric} to \eqref{eqn rank 1 central coefficient written as in rank 2},
we have 
\begin{align*}
    I_{0}(T;\ell)=(2\pi)^{-2\ell}\sum_{k=0}^{\ell}\frac{(-B)^{k}2^{-2k}(2\ell+2k)!}{(k!)^{2}(\ell+k)!(\ell-k)!}\int_{\R}\frac{e^{4\pi i x}dx}{(x^{2}+B)^{\ell+k+1/2}}.
\end{align*}
Using Basset's integral,
one has 
\[\int_{\R}\frac{e^{iwt}dt}{(t^{2}+z^{2})^{n+1/2}}=\frac{2\sqrt{\pi}}{\Gamma(n+1/2)}\left(\frac{|w|}{2|z|}\right)^{n}K_{n}(|wz|),\,w,z\in \R.\]
Plugging this into $I_{0}(T;\ell)$,
we obtain that:
\begin{align*}
    I_{0}(T;\ell)&=(2\pi)^{-2\ell}\sum_{k=0}^{\ell}\frac{(-B)^{k}2^{-2k}(2\ell+2k)!}{(k!)^{2}(\ell+k)!(\ell-k)!}\cdot \frac{2^{2l+2k+1}(\ell+k)!B^{-(\ell+k)/2}(2\pi)^{\ell+k}K_{\ell+k}(4\pi\sqrt{B})}{(2\ell+2k)!}\\
    &=2\pi^{-2\ell} B^{-\ell}\sum_{k=0}^{\ell}\frac{(-1)^{k}}{(k!)^{2}(\ell-k)!}(2\pi\sqrt{B})^{\ell+k}K_{\ell+k}(4\pi \sqrt{B}).
\end{align*}
\begin{lemma}\label{lemma summation of Bessel functions}
    For any $C>0$,
    we have 
    \[S_{\ell}:=\sum_{k=0}^{\ell}\frac{(-1)^{k}C^{\ell+k}}{(k!)^{2}(\ell-k)!} K_{\ell+k}(2C)=\frac{(-1)^{\ell}C^{2\ell}}{(\ell!)^{2}}K_{0}(2C).\]
\end{lemma}
\begin{proof}
    Using the recurrence relation $K_{n+1}(z)=K_{n-1}(z)+\frac{2n}{z}K_{n}(z)$ for K-Bessel functions,
    we can write $-C^{2}S_{\ell}$ as 
    \begin{align*}
        &\sum_{k=0}^{\ell}(-1)^{k+1}\frac{C^{\ell+k+2}}{(k!)^{2}(\ell-k)!}K_{\ell+k}(2C)\\
        =&\sum_{j=1}^{\ell+1}(-1)^{j}\frac{C^{\ell+1+j}}{((j-1)!)^{2}(\ell-j+1)!}\left(K_{\ell+1+j}(2C)-\frac{\ell+j}{C}K_{\ell+j}(2C)\right)\\
        =&\sum_{j=1}^{\ell+1}\frac{(-1)^{j}C^{\ell+1+j}K_{\ell+1+j}(2C)}{((j-1)!)^{2}(\ell-j+1)!}+\sum_{j=0}^{\ell}\frac{(-1)^{j}(\ell+j+1)C^{\ell+1+j}K_{\ell+j+1}(2C)}{(j!)^{2}(\ell-j)!}\\
        =&\frac{(\ell+1)C^{\ell+1}K_{\ell+1}(2C)}{\ell!}+\sum_{j=1}^{\ell}\frac{(-1)^{j}C^{\ell+1+j}K_{\ell+1+j}(2C)}{(j!)^{2}(\ell+1-j)!}(\ell+1)^{2}+(-1)^{\ell+1}\frac{C^{2\ell+2}K_{2\ell+2}(2C)}{(\ell!)^{2}}\\
        =&(\ell+1)^{2}S_{\ell+1}.
    \end{align*}
    Since $S_{0}=K_{0}(2C)$ and $S_{\ell+1}=-\frac{C^{2}}{(\ell+1)^{2}}S_{\ell}$,
    we have the desired result.
\end{proof}
By this lemma,
one has 
\[I_{0}(T;\ell)=2\pi^{-2\ell} B^{-\ell}\cdot\frac{(-1)^{\ell}(2\pi \sqrt{B})^{2\ell}}{(\ell!)^{2}}K_{0}(4\pi\sqrt{B})=\frac{(-1)^{\ell}2^{2\ell+1}}{(\ell!)^{2}}K_{0}(4\pi \sqrt{B}).\]
As a consequence 
\[E_{1,T}^{\infty}=(-1)^{\ell}\frac{2^{2\ell+1}}{(\ell!)^{2}}\call{W}_{T}(1).\]
Combining this result with \Cref{prop non-arch rank 1 Fourier},
we have the following proposition:
\begin{prop}
    \label{prop rank 1 Fourier coefficient}
    Suppose that $\ell>n$. 
    For a non-zero isotropic vector $T\in \bff{V}_{0}$ and $g_{\infty}\in \bff{G}(\R)$,
    if $T\notin\bff{V}_{0}(\Z)$, the Fourier coefficient $E_{\ell}^{T}(g_{\infty},s=\ell+1)$ is zero;
    if $T\in \bff{V}_{0}(\Z)$, we have:
    \begin{align*}
        E_{\ell}^{T}(g_{\infty},\ell+1)=\frac{(-1)^{\ell}2^{2\ell+1}\sigma_{E,\ell}(T)}{(\ell!)^{2}}\call{W}_{T}(g_{\infty}).
    \end{align*}
    
    In particular, for any non-zero vector $T\in \bff{V}_{0}$ with $\lrangle{T}{T}=0$ 
    and any $g=g_{f}g_{\infty}\in \bff{M}(\A_{f})\times\bff{G}(\R)$,
    $E_{\ell}^{T}(g_{f}g_{\infty},\ell+1)$ is a rational multiple of $\call{W}_{T}(g_{\infty})$.
\end{prop}
\subsection{Computation of constant term}
\label{section constant term of Eisenstein series}
The constant term of $E_{\ell}(g,s=\ell+1)$ consists of three parts:
\begin{itemize}
    \item $E_{\ell,0}(g,s=\ell+1)=f_{\ell}(g,s=\ell+1)$,
    \item $E_{\ell,1}^{0}(g,s=\ell+1)=\sum\limits_{L\in\call{L}_{0}}\int_{\bff{N}_{L}(\A)\backslash \bff{N}(\A)}f_{\ell}(\gamma(L)ng,s=\ell+1)dn$,
    \item $E_{\ell,2}^{0}(g,s=\ell+1)=\int_{\bff{N}(\A)}f_{\ell}(w_{2}ng)dn$.
\end{itemize}
In this section, we show that the last two terms are zero, hence the constant term equals $f_{\ell}(g,\ell+1)$.
\subsubsection{\texorpdfstring{$E_{\ell,0}(g,s=\ell+1)$}{}}
\label{section i=0 part constant term}
\begin{lemma}
    \label{lemma value of constant term}
    For $g\in \bff{P}(\A)$,
    \[E_{\ell,0}(g,s)=f_{\ell}(g,s)=\pi^{-2\ell-1}|\nu(g)|_{\A_{E}}^{s}\zeta_{E}(s)[u_{1}^{\ell}][u_{2}^{\ell}].\]
\end{lemma}
\begin{proof}
    For $g_{f}\in \bff{P}(\A_{f})$,
    we have 
    \begin{align*}
        f_{fte}(g_f,s)=\int_{\A_{E,f}^{\times}}|t|_{\A_{E,f}}^{s}\Phi_{f}(tb_{1}g_{f})dt
        =|\nu(g_{f})|_{\A_{E,f}}^{s}\int_{\A_{E,f}^{\times}}|t|_{\A_{E,f}}^{s}\Phi_{f}(tb_{1})dt.
    \end{align*}
    Thus, the non-archimedean contribution is $|\nu(g_{f})|_{E}^{s}\zeta_{E}(s)$.
    Combining this with
    \[f_{\ell,\infty}(g_{\infty},s)=\pi^{-2\ell-1}|\nu(g_{\infty})|_{\C}^{s}[u_{1}^{\ell}][u_{2}^{\ell}],\]
    we obtain the desired identity.
\end{proof}
\subsubsection{\texorpdfstring{$E_{\ell,1}^{0}(g,s=\ell+1)$}{}}
\label{section i=1 part constant term}
Let $L_{0}$ be the isotropic line spanned by $c_{1}\in \bff{V}_{0}$.
Define $\bff{P}_{0}$ to be the stabilizer of $L_{0}$ in $\bff{U}(\bff{V}_{0})$,
which is a parabolic subgroup of $\bff{M}$.
Denote the similitude character of $\bff{P}_{0}$ by $\lambda$,
\emph{i.e.\,}$c_{1}g=\lambda(g)^{-1}v_{0}$ for any $g\in\bff{P}_{0}$.
We rewrite $E_{\ell,1}^{0}(g,s)$ as:
\begin{align*}
    E_{\ell,1}^{0}(g,s)&=\sum_{L\in \call{L}_{0}}\int_{\bff{N}_{L}(\A)\backslash \bff{N}(\A)}f_{\ell}(\gamma(L)ng,s)dn
    =\sum_{\gamma\in \bff{P}_{0}(\Q)\backslash\bff{M}(\Q)}\int_{\bff{N}_{L_{0}}(\A)\backslash \bff{N}(\A)}f_{\ell}(w_{1}\gamma ng,s)dn.
\end{align*}
Set $f_{0}(g,s)=\int_{\bff{N}_{L_{0}}(\A)\backslash \bff{N}(\A)}f_{\ell}(w_{1}ng,s)dn,$
then for $\rmm{Re}(s)\gg 0$,
\[E_{\ell,1}^{0}(g,s)=\sum_{\gamma\in\bff{P}_{0}(\Q)\backslash \bff{M}(\Q)}f_{0}(\gamma g,s),\]
and when restricted to $g\in \bff{M}(\A)$ it defines an Eisenstein series on $\bff{M}$.
Since $f_{0}$ is defined by an Eulerian integral,
we consider its archimedean part 
\[f_{0,\infty}(g,s)=\int_{\bff{N}_{L_{0}}(\R)\backslash\bff{N}(\R)}f_{\ell,\infty}(w_{1}ng,s)dn.\]
The following result shows that when $s=\ell+1$ this section is identically zero,
thus the Eisenstein series associated with $f_{0}$ has no contribution to the constant term of $E_{\ell}(g,s=\ell+1)$:
    \begin{prop}\label{prop vanishing of i=1 contribution to constant term}
    For any $g\in \bff{M}(\R)$, $f_{0,\infty}(g,s=\ell+1)=0$.
   \end{prop}
   \begin{proof}
    We identify $\bff{N}_{L_{0}}(\R)\backslash \bff{N}(\R)$ with $\C$ via the map $n(v,\lambda)\mapsto-\lrangle{c_{1}}{v}$,
    and for any $z\in\C$, we choose an element $n(z)\in \bff{N}(\R)$ under this identification.

    It suffices to check the vanishing for $g\in \bff{P}_{0}(\R)$.
    Similarly to \Cref{lemma rank 1 archimedean section},
    one has 
    \[\pi^{-2\ell-1}f_{\ell,\infty}(w_{1}n(z)g)=\frac{(z\nu(g)^{-1}u_{1}+\lambda(g)^{-1}u_{2})^{\ell}(-\overline{\lambda(g)}^{-1}u_{1}+\overline{z\nu(g)^{-1}}u_{2})^{\ell}}{(|\lambda(g)|^{-2}+|z|^{2}|\nu(g)|^{-2})^{2\ell+1}(\ell!)^{2}}.\]
    From now on we write $\lambda(g)$ (\emph{resp.\,}$\nu(g)$) as $\lambda$ (\emph{resp.\,}$\nu$).

    The coefficient of $[u_{1}^{\ell-v}][u_{2}^{\ell+v}]$ in $\pi^{-2\ell-1}f_{\ell,\infty}(w_{1}n(z)g)$ is 
    \begin{align*}
        &\frac{(\ell-v)!(\ell+v)!}{(\ell!)^{2}(|\lambda|^{-2}+|z|^{2}|\nu|^{-2})^{2\ell+1}}\sum_{\substack{0\leq i,j\leq \ell\\ i+j=\ell-v}}\binom{\ell}{i}\binom{\ell}{j}z^{i}\nu^{-i}\lambda^{-(\ell-i)}(-\overline{\lambda})^{-j}(\overline{z\nu^{-1}})^{\ell-j}\\
        =&\frac{(\ell-v)!(\ell+v)!|\nu|^{-2\ell}\nu^{v}\lambda^{-v}|z|^{2\ell}z^{-v}}{(\ell!)^{2}(|\lambda|^{-2}+|z|^{2}|\nu|^{-2})^{2\ell+1}}\sum_{j=\max(0,-v)}^{\min(\ell,\ell-v)}(-1)^{j}\binom{\ell}{j}\binom{\ell}{j+v}|\nu|^{2j}|\lambda|^{-2j}|z|^{-2j},
    \end{align*}
    and the corresponding coefficient in $\pi^{-2\ell-1}f_{0,\infty}(g,s)$ is the integral of this function over $\C$,
    with respect to the area element $d_{A}z:=dxdy=\frac{i}{2}dzd\overline{z}$.
    Set $w=z\lambda/\nu$,
    then we have $d_{A}w=|\lambda/\nu|^{2}d_{A}z$ 
    and this function equals 
    \[\frac{(\ell-v)!(\ell+v)!|\lambda|^{2\ell+2}w^{-v}|w|^{2\ell}}{(\ell!)^{2}(1+|w|^{2})^{2\ell+1}}\sum_{j=\max(0,-v)}^{\min(\ell,\ell-v)}(-1)^{j}\binom{\ell}{j}\binom{\ell}{j+v}|w|^{-2j}.\]
    Hence, the coefficient of $[u_{1}^{\ell-v}][u_{2}^{\ell+v}]$ in $\pi^{-2\ell-1}f_{0,\infty}(g,s=\ell+1)$ equals
    \begin{equation}
        \frac{(\ell-v)!(\ell+v)!}{(\ell!)^{2}}|\lambda|^{2\ell}|\nu|^{2}\sum_{j=\max(0,-v)}^{\min(\ell,\ell-v)}(-1)^{j}\binom{\ell}{j}\binom{\ell}{j+v}\left(\int_{\C}\frac{w^{-v}|w|^{2\ell-2j}}{(1+|w|^{2})^{2\ell+1}}d_{A}w\right)
    \end{equation}
    Notice that the inner integral is nonzero only when $v=0$,
    and we have 
    \[\int_{\C}\frac{|w|^{2\ell-2j}}{(1+|w|^{2})^{2\ell+1}}d_{A}w=\pi\frac{(\ell-j)!(\ell+j-1)!}{(2\ell)!},\,j=0,1,\ldots,\ell.\]
    So we only need to consider the coefficient of $[u_{1}^{\ell}][u_{2}^{\ell}]$,
    which is equal to 
    \begin{align*}
        \frac{\pi|\lambda|^{2\ell}|\nu|^{2}}{(2\ell)!}\sum_{j=0}^{\ell}(-1)^{j}\binom{\ell}{j}^{2}(\ell-j)!(\ell+j-1)!
        =\frac{\pi (\ell!)^{2} |\lambda|^{2\ell}|\nu|^{2}}{\ell(2\ell)!}{_{2}F_{1}}(\ell,-\ell;1;1).    
    \end{align*}
    The hypergeometric value ${_{2}F_{1}}(\ell,-\ell;1;1)$ equals $0$ when $\ell\geq 1$.
   \end{proof}
  
\subsubsection{\texorpdfstring{$E_{\ell,2}^{0}(g,s=\ell+1)$}{}}
\label{section i=2 part constant term}
The vanishing of this term when $\ell>n$ can be proved either by generalizing our calculation in \Cref{section archimedean place rank 2 Fourier infinite part} to the coefficient of any $[u_{1}^{\ell-v}][u_{2}^{\ell+v}]$ in $E_{2,\infty}^{T}$,
or by a similar argument to \cite[Proposition 4.1.2]{Po}.
Here we choose to give a ``softer'' proof:
\begin{prop}\label{prop no contribution of i=2 to constant term}
    When $\ell>n$, 
    $E_{\ell,2}^{0}(g,s=\ell+1)$ is identically zero. 
\end{prop}
\begin{proof}
    It suffices to prove the vanishing of $E_{\ell,2,\infty}^{0}(m,s=\ell+1)$ for $m\in \bff{M}(\R)$.
    We first look at the archimedean part $E_{\ell,\infty}^{0}(m,s=\ell+1)$ of the constant term $E_{\ell}^{0}(m,s=\ell+1)$,
    and write it as 
    \[E_{\ell,\infty}^{0}(m,s=\ell+1)=\sum_{v=-\ell}^{\ell}f_{v}(m)[u_{1}^{\ell-v}][u_{2}^{\ell+v}].\]
    As in the proof of \Cref{prop at some point Eisenstein series is modular form},
    we write $m=(z,h)\in \bff{M}(\R)$ as $m=(h,r,\theta)$ so that $z=re^{i\theta},\,r\in\R_{>0}$ and $\theta\in[0,2\pi)$. 
    For any $r\in \R_{>0}$,
    denote by $m(r)$ the element $(\mathrm{id},r,0)\in\bff{M}(\R)$.
    Applying \cite[Proposition 3.10]{HMY} to $E_{\ell,\infty}^{0}$,
    we have:
    \begin{align*}
        (r\partial r -2(\ell+1)-v)f_{v}(m)&=0,\,v=-\ell,\ldots,\ell-1,\\
        (r\partial r -2(\ell+1)+v)f_{v}(m)&=0,\,v=-\ell+1,\ldots,\ell.
    \end{align*}
    These equations imply that $f_{v}\equiv 0$ except for $v=0$ or $\pm \ell$,
    and the three non-zero coefficients satisfy
    $(r\partial r-2(\ell+1))f_{0}=0$ and $(r\partial r-\ell-2)f_{\pm\ell}=0$.
    Hence, $f_{0}$ (\emph{resp.\,}$f_{\pm \ell}$) is the product of a function of $h,\theta$ with $r^{2\ell+2}$ (\emph{resp.\,}$r^{\ell+2}$).
    Combining this with \Cref{lemma value of constant term} and \Cref{prop vanishing of i=1 contribution to constant term},
    the coefficients of $E_{\ell,2,\infty}^{0}$ transform by $r^{2\ell+2}$ and $r^{\ell+2}$ under the left translation $m\mapsto m(r)m$ for $r>0$.

    On the other hand, by \eqref{eqn tranform of rank 2 Fourier},
    we have 
    \[E_{\ell,2,\infty}^{0}(m(r)m,s=\ell+1)=r^{2n-2\ell}\cdot E_{\ell,2,\infty}^{0}(m,s=\ell+1).\]
    Thus $E_{\ell,2,\infty}^{0}(g,s=\ell+1)$ must transform by $r^{2n-2\ell}$.
    However, when $\ell>n$,
    $2n-2\ell$ is equal to neither $\ell+2$ nor $2\ell+2$,
    so $E_{\ell,2,\infty}^{0}(m,s=\ell+1)$ is identically zero on $\bff{M}(\R)$.
\end{proof}
In conclusion,
we have shown the following:
\begin{prop}\label{prop constant term}
    For any $g\in \bff{P}(\A)$, the constant term $E^{0}(g,s=\ell+1)$ of the Eisenstein series $E_{\ell}(g,s=\ell+1)$ is equal to 
    \[f_{\ell}(g,s=\ell+1)=|\nu(g)|_{\A_{E}}^{\ell+1}\frac{\zeta_{E}(\ell+1)}{\pi^{2\ell+1}}[u_{1}^{\ell}][u_{2}^{\ell}].\]
\end{prop}

\subsection{The full Fourier expansion of \texorpdfstring{$E_{\ell}(g,s=\ell+1)$}{PDFstring}}
\label{section full Fourier expansion}
Combining \Cref{thm rank 2 Fourier coefficient}, \Cref{prop rank 1 Fourier coefficient} and \Cref{prop constant term} together,
we obtain our main theorem:
\begin{thm}\label{thm full Fourier expansion}
Suppose $\ell> n$, $D\equiv 3\modulo 4$ and $n\equiv 2\modulo 4$. 
The constant term along $\bff{Z}$ of the Eisenstein series   
$E_\ell(g, s=\ell+1)$ has the Fourier expansion 
\begin{align*}
   E_{\ell,0}(g)+\sum_{\substack{0\neq T\in\bff{V}_{0}\\\lrangle{T}{T}\geq 0}}a_{T}(g_{f})\call{W}_{T}(g_{\infty}),\,g=g_{f}g_{\infty}\in\bff{G}(\A_{f})\times \bff{G}(\R),
\end{align*}
where the Fourier coefficients satisfy:
\begin{itemize}
    \item The locally constant functions $a_{T}(g_{f})$ are valued in $\Q$ when restricted to $g_{f}\in\bff{M}(\A_{f})$.
    \item The $\bff{N}$-constant term $E_{\ell,0}(g)$ is a rational multiple of $|\nu(g)|_{\A_{E}}^{\ell+1}\frac{\zeta_{E}(\ell+1)}{\pi^{2\ell+1}}[u_{1}^{n}][u_{2}^{n}]$ when restricted to $g\in \bff{M}(\A)$.
\end{itemize}
Moreover, when $g_{f}=1$, 
$E_{\ell}(g_{\infty},s=\ell+1)$ have the following explicit Fourier expansion:
\begin{align*}
   E_{\ell,0}(g_\infty)+C_{\ell}\sum_{0\neq T\in\bff{V}_{0}(\Z)\atop \lrangle{T}{T}=0}\sigma_{E,\ell}(T)\call{W}_{T}(g_{\infty})
    +D_{n,\ell}\sum_{T\in\bff{V}_{0}(\Z)\atop \lrangle{T}{T}>0}\left(\prod_{p|\lrangle{T}{T}}Q_{T,p}(p^{\ell-\frac{n-1}{2}})\right)\call{W}_{T}(g_{\infty}),
\end{align*}
where
\[C_{\ell}=\frac{(-1)^{\ell}2^{2\ell+1}}{(\ell!)^{2}}\quad \text{and} \quad D_{n,\ell}=\frac{2^{2n+2}(2\ell-n+2)D^{\ell+1-\frac{n}{2}}}{(\ell!)^{2}\left|B_{2\ell-n+2}\right|\sigma_{\ell+1-\frac{n}{2}}(D)} \in \Bbb Q.\]
\end{thm}

\section{Conjectural Saito-Kurokawa type lift}\label{Saito-Kuro}
\label{section conj Saito Kurokawa lift}
For a prime $p$ and a vector $T\in \bff{V}_{0}(\Z)$ with $\lrangle{T}{T}>0$ and $p^{k}\| \langle T,T\rangle$, 
define \[\widetilde Q_{T,p}(X):=X^{-k}Q_{T,p}(X)\in\Z[p^{1/2}][X,X^{-1}],\] 
which is a Laurent polynomial satisfying the functional equation $\widetilde Q_{T,p}(X^{-1})=\widetilde Q_{T,p}(X)$.
The $T$th Fourier coefficient of the Eisenstein series $E_\ell(g_{\infty},s=\ell+1)$ at $g_{\infty}\in \bff{G}(\R)$ is 
$$
D_{n,\ell} \lrangle{T}{T}^{\ell-\frac{n-1}{2}}\prod_{p} \widetilde Q_{T,p}(p^{\ell-\frac{n-1}{2}})\call{W}_{T}(g_{\infty}).
$$

Let $h$ be a cuspidal Hecke eigenform on $\SL_{2}(\Z)$ with weight $2\ell-n+2$,
and $\pi_{h}$ the automorphic representation of $\GL_{2}$ associated with $h$.
Let $\rmm{diag}(\alpha_{p},\alpha_{p}^{-1})$ be the Satake parameter of $\pi_{h,p}$.
\begin{defi}
    \label{def coeff of lift}
    For any $T\in \bff{V}_{0}$ with $\lrangle{T}{T}>0$ and $g_{f}=n_{f}m_{f}k_{f}\in\bff{G}(\A_{f})$, where $n_{f}\in \bff{N}(\A_{f})$, $m_{f}=(z_{f},h_{f})\in \bff{M}(\A_{f})$ and $k_{f}$ lies in the stabilizer of $\widehat{\Z}b_{1}\oplus\bff{V}_{0}(\widehat{\Z})\oplus \widehat{\Z}b_{2}$ in $\bff{G}(\A_{f})$,
    define 
    \[A_{h}(T)(g_{f}):=|z_{f}|_{\A_{E,f}}^{\frac{n+1}{2}}\chi_{T}(n_{f})\lrangle{T}{T}^{\ell-\frac{n-1}{2}}\prod_{p}\widetilde{Q}_{z_{p}T\cdot h_{p},p}(\alpha_{p}).\]
\end{defi}
Let $F_h$ be the function on $\bff G(\Bbb A)$ whose constant term along $Z$ has the Fourier expansion:
\begin{align*}
    F_{h,Z}(g)=\sum_{\substack{T\in \bff{V}_{0}\\ \lrangle{T}{T}>0}} A_h(T)(g_{f})\call{W}_{T}(g_\infty),\text{ for any }g=g_f g_\infty \in G(\A).    
\end{align*}
Following \cite{Ik01, KY2}, we make the following conjecture:
\begin{conj}\label{conj modularity of lift}
    The function $F_h$ is a quaternionic cuspidal Hecke eigenform of weight $\ell$ on $\bff{G}$.
\end{conj}
We expect $F_h$ to be the same as the one in \cite[Theorem 5.4]{HMY} constructed by the theta lift from $\rmm{U}(1,1)$.
An evidence is that admitting \Cref{conj modularity of lift},
the standard $L$-function of the cuspidal automorphic representation $\Pi_{F_h}$ of $\rmm{U}(2,n)$ generated by $F_{h}$ satisfies:
\[\rmm{L}(s,\Pi_{F_{h}},\rmm{St})=\rmm{L}(s,\rmm{BC}(\pi_{h}))\prod_{i=0}^{n-1}\zeta_{E}\left(s+\frac{n-1}{2}-i\right),\]
where $\rmm{BC}(\pi_h)$ is the base change of $\pi_h$ to $E$. 
This $L$-function is obtained by the calculation in \cite[\S 2.1.2]{KK} and the fact that $\Pi_{F_{h},p}\simeq \rmm{Ind}_{\bff{P}(\Q_{p})}^{\bff{G}(\Q_{p})}|\nu|_{E_{p}}^{s_{p}}\delta_{\bff{P}}^{\frac{1}{2}}$,
where $p^{s_{p}}=\alpha_{p}$.
This shows that the (conjectural) Arthur parameter of $\Pi_{F_{h}}$ is the sum of the Arthur parameter of the transfer $\widetilde{\pi_{h}}$ of $\pi_{h}$ to $\rmm{U}(1,1)$ and the irreducible $n$-dimensional representation of Arthur's $\SL_{2}(\C)$,
which matches Adam's conjecture \cite[Conjecture 4.1]{Gan} on how Arthur parameter behaves under the theta correspondence.

Moreover,
$\Pi_{F_{h}}$ is a CAP representation in a general sense, namely,
for almost all finite place $p$, $\Pi_{F_{h},p}$ is isomorphic to the $\Q_{p}$-component of the CAP representation 
\[\rmm{Ind}_{\bff{Q}}^{\bff{G}^{*}}\left(|\cdot|^{\frac{n-1}{2}}_{\A_{E}}\otimes|\cdot|_{\A_{E}}^{\frac{n-3}{2}}\otimes\cdots\otimes|\cdot|_{\A_{E}}^{\frac{1}{2}}\otimes \widetilde{\pi_{h}}\right)\otimes\delta_{\bff{Q}}^{\frac{1}{2}}\]
of the quasi-split inner form $\bff{G}^{*}=\rmm{U}(\frac{n+2}{2},\frac{n+2}{2})$ of $\bff{G}$,
where $\bff{Q}$ is a parabolic subgroup whose Levi factor is isomorphic to $(\rmm{Res}_{E/\Q}\GL_{1})^{\frac{n}{2}}\times \rmm{U}(1,1)$.

\end{document}